\documentclass[12pt]{article}

\usepackage[utf8]{inputenc}
\usepackage[english]{babel}
\usepackage{amsmath,amsthm,amsfonts}
\usepackage{graphicx}
\usepackage{bm}
\usepackage{tikz}
\usepackage{multirow}
\usepackage[makeroom]{cancel}

\newtheorem{theorem}{Theorem}

\newtheorem{lemma}{Lemma}
\newtheorem{remark}{Remark}
\graphicspath{ {./}}

\renewcommand{\div}{\mathop{\rm div}\nolimits}

\begin{document}
\title{Coupling of multiscale and multi-continuum approaches}

\author{
Eric T. Chung \thanks{Department of Mathematics,
The Chinese University of Hong Kong (CUHK), Hong Kong SAR. Email: {\tt tschung@math.cuhk.edu.hk}.
The research of Eric Chung is supported by Hong Kong RGC General Research Fund (Project 400411).}
\and
Yalchin Efendiev \thanks{Department of Mathematics \& Institute for Scientific Computation (ISC),
Texas A\&M University,
College Station, Texas, USA. Email: {\tt efendiev@math.tamu.edu}.}
\and
Tat Leung \thanks{Department of Mathematics, Texas A\&M University, College Station, TX 77843} \and
 Maria Vasilyeva\thanks{Department of Computational Technologies, Institute of Mathematics and Informatics, North-Eastern Federal University, Yakutsk, Republic of Sakha (Yakutia), Russia, 677980 \& Institute for Scientific Computation, Texas A\&M University, College Station, TX 77843-3368}
}

\newcommand{\hg}[1]{{\textcolor{green} {#1}}}

\newcommand{\hb}[1]{{\textcolor{blue} {#1}}}

\maketitle
\begin{abstract}
Simulating complex processes in fractured media requires some type of model reduction.
Well-known approaches include multi-continuum techniques, which have been commonly used
in approximating subgrid effects for flow and transport in fractured
media.
Our goal in this paper is to (1)  show
a relation between multi-continuum approaches and
Generalized Multiscale Finite Element Method (GMsFEM) and
(2) to discuss coupling these approaches for solving problems in complex
multiscale fractured media.
The GMsFEM, a systematic approach,
constructs multiscale basis functions via local spectral
decomposition in pre-computed snapshot spaces.
We show that
GMsFEM can automatically identify separate
fracture networks  via local spectral problems. We discuss
the relation between these basis functions and continuums in
multi-continuum methods. {The GMsFEM can}
automatically detect each {continuum} and represent the interaction
between the {continuum} and its surrounding (matrix).
For problems with simplified fracture networks, we propose
a simplified basis construction with the GMsFEM. This
simplified approach
is effective {when the fracture networks are known and have simplified
geometries}. We
show that {this approach} can achieve a similar result compared to the results
using the GMsFEM with spectral basis functions.
Further, we discuss {the} coupling between the GMsFEM and
multi-continuum approaches. In this case, many fractures are
resolved while for unresolved fractures, we use
a multi-continuum approach with local Representative Volume Element (RVE) information.
As a result, {the method} deals with a system of equations
on a coarse grid, where each equation represents one of the continua
on the fine grid. We present various basis construction mechanisms
and numerical
results. The GMsFEM framework,
in addition, can provide adaptive and online
basis functions
to improve the accuracy of coarse-grid simulations.
These are discussed in the paper. In addition, we present an example
of the application of our approach to shale gas transport in fractured media.

\end{abstract}

\section{Introduction}

{\bf Multiscale phenomena in fractured media.}
Subsurface formations with discrete fractures, faults, thin features
are common in many applications.
These include fractured subsurface formations,
fractured composite materials, and so on. A main challenge in
simulating complex processes is
due to multiple scale and high contrast.
The material properties within fractures
can be very different from the background properties. Due to complex
fracture configuration{s}, there are multiple scales and high contrast.

{\bf Fine-grid simulation for fractures.}
Constructing a fine-grid simulation model
 is typically done in several steps (we refer \cite{karimi2016general}
for the overview). As a first step,
an unstructured grid is used to describe the
fractures.
Then, the flow/transport equations are discretized on the unstructured
grid.
A variety of techniques have been applied for flow simulation
in porous media with discrete fractures using both finite-element
and finite-volume methods. Within the finite-element framework,
the standard Galerkin formulation (\cite{baca1984modelling, juanes2002general, karimi2003numerical, kim2000finite}),
the mixed finite-element method (\cite{erhel2009flow, hoteit2008efficient, ma2006mixed, martin2005modeling}){,} and the
discontinuous Galerkin method (\cite{eikemo2009discontinuous, hoteit2005multicomponent}) have been used to simulate single-phase and
multiphase flow in discrete fracture models. Within the finite-volume
framework, formulations have been presented by, e.g., \cite{bogdanov2003two, granet2001two, karimi2004efficient, monteagudo2004control, noetinger2015quasi, reichenberger2006mixed}. A hybrid approach combining the
finite-element method for the pressure equation and the finite volume
method for transport has also been investigated (\cite{geiger2009black, matthai2007finite, nick2011comparison}).

{\bf The need for model reduction.} Because of multiple scales
and high contrast,
some type of model reduction is needed for simulating physical
processes in fractured media. Typical approaches divide the domain
into coarse grids, where effective properties in each coarse-grid block
are computed \cite{dur91,weh02}.
The standard upscaling methods compute the effective properties
using the solution of local problems in each coarse block or
representative
volume. However, it is known that these approaches
are not sufficient as each coarse block contains multiple
important modes.
This led to multi-continuum approaches {\cite{arbogast1990derivation,barenblatt1960basic,kazemi1976numerical,pruess1982fluid,warren1963behavior,wu1988multiple}},
where several equations
are formulated for each coarse block. In particular, the flow equation
for  the background
(called matrix) and the fracture are written separately with
some interaction terms.
These approaches make several assumptions such as each continua
connected throughout the domain and the form of the coupling.
Our goal is to show a relation between these approaches and
some multiscale finite element methods and further discuss generalizations
based on these approaches.

{\bf Brief introduction to the GMsFEM.}
The GMsFEM {\cite{chung2016adaptive,chung2015generalizedwave,egh12}} follows the framework of the Multiscale Finite Element Method (MsFEM) and
is introduced to systematically add new degrees of freedom in each coarse
block. The new basis functions are computed by constructing the snapshots
and performing local spectral decomposition in the snapshot space.
It was shown that there is a spectral gap and the eigenvectors
corresponding to very small eigenvalues represent the connected
high-conductivity networks. These dominant eigenvectors
represent fracture networks and can be thought as reduced degrees of
freedom representing each continua as explained later.

{\bf This paper. }
In this paper, we discuss {\it a relation between the GMsFEM and the
multi-continuum approaches}.
As mentioned, the dominant eigenvectors
represent the connected fracture networks. For example,
if there are $n$ separate fracture networks within a coarse block, then
we will have $n$ very small eigenvalues and the corresponding
eigenvectors represent these fracture networks. We {give} a detailed
comparison between the multi-continua approaches and the GMsFEM in the paper.
We discuss the interaction between different
continuum media.
In multi-continuum approaches, this interaction is modeled
based on physical principles, while the GMsFEM approach
provides a rigorous coupling between multiple continua.
We note the GMsFEM automatically takes into account the interaction
of various continua between different coarse blocks, while
in multi-continuum approaches, this interaction is
stated apriori based on physical principles \cite{yan2016beyond}.
We also present simplified basis functions
that are related to fractures if the fracture networks are identified.


In the paper, we discuss {\it a coupled GMsFEM and
 the multi-continuum approaches} by considering fractures
over a very rich hierarchy of scales.
Our approach uses multi-continuum at the fine grid
and the GMsFEM for modeling the fractures that can be resolved
on the fine grid (see Figure \ref{schematic}).
In this case, {the method} deals with a system of equations coupled
with the fracture network. First, we discuss a GMsFEM for
this system. Secondly, we discuss approaches for computing
the parameters of the multi-continuum system based on local
Representative Volume Element (RVE) computations.
We discuss the setup of local RVE problems and
compute the parameters for the multi-continuum fine-grid
discretization. Since these parameters, in general, are heterogeneous,
{the use of coupled basis functions is crucial}, {and their constructions will be presented}.
{We will also} discuss a
relation between the GMsFEM and the Multiple
Interacting Continua (MINC)
\cite{pruess1982fluid}.

{On the other hand},
we {establish} {\it a relation between offline GMsFEM and the
multi-continua approaches.} The GMsFEM has several important
fundamental ingredients that can further be used to achieve higher accuracy
and more efficiency.
The first ingredient includes the adaptivity. The GMsFEM{’s}
adaptivity can be used to add multiscale basis functions in selected regions.
This concept can be effectively used to add new multiscale
basis functions in selected regions.
The second ingredient of the GMsFEM
 is online basis functions. These basis functions
are constructed using the residual information (adaptively in space and time)
 to speed-up the simulations. This can be used
to speed-up the convergence of the proposed method.

We {will show} numerical results.
First, we discuss the GMsFEM{’s} basis construction
and numerically show how to identify the number of continua
based on local spectral decomposition and the spectrum.
Then, we present
a simplified basis construction and  numerical results
for the GMsFEM using both simplified basis construction and
a general approach.
In the second part, we {demonstrate} numerical results
when the GMsFEM and the multi-continuum approaches are coupled.
In this case, the multi-continuum approach is used on the fine grid.
Our numerical results use both coupled and un-coupled basis functions
and show that the GMsFEM {is able to couple with the multi-continuum appraoch and gives accurate solution using few basis functions}. The GMsFEM can be used
for heterogeneously varying multi-continuum problems.
In \cite{aevw16}, we have applied the GMsFEM to shale gas
flows in multi-continuum
media. In this paper, we also present an example of the application
of our proposed approach to shale gas transport.

{Furthermore}, we will present an analysis for the GMsFEM when the fine-scale
problem is described by a multi-continuum approach. In this case, {the method gives}
a system of coupled equations. We study the convergence of
the GMsFEM for cases when basis functions are  independently constructed
and in a coupled fashion. In both cases, convergence results are obtained.

The paper is organized as follows. In Section \ref{sec:sec1},
we discuss the relation between the GMsFEM and multi-continuum
approaches. We {develop} simplified basis functions and
show numerical results. In Section \ref{sec:sec2},
we discuss the coupled GMsFEM and the multi-continuum
approach. We {also show} numerical results in this section.
The analysis is {given} in the Appendix \ref{appendix}.

\section{Relation between the GMsFEM and multi-continuum}
\label{sec:sec1}

The goal of this section is to highlight the similarities
between the multi-continuum approaches and the GMsFEM.
We assume that the fractures are resolved
on a fine grid.
We show that (1) the GMsFEM can identify fracture
networks and result {in} a similar system as a multi-continuum
approach{,} (2) the GMsFEM can resolve the detailed fracture
and matrix interaction{, and} (3) the GMsFEM basis functions
can be computed in a simplified way.

\subsection{Fine-grid equations in fractured media}
We consider a detailed fine-grid discretization of the flow equation
in the fractured media
\begin{equation}
  \label{eq:main1}
  c {\partial u\over\partial t} = {\text{div}}(\kappa \nabla u) +q,
  \end{equation}
where $u$ is the solution, $q$ is the source term, $\kappa$ is permeability and $c$ is porosity.
The permeability is large within fractures and the porosity
has a smaller value in the fractures.
{Fractures} are modeled
as one dimensional objects.

The domain $D$ is divided into the fracture and the matrix region
\begin{equation}
D = D_{m} \oplus_i \, d_i D_{f,i},
\end{equation}
where $m$ and $f$ represent the matrix and the fracture
regions.
$d_i$ denotes the aperture of the $i$-th fracture and
$i$ is the index of the fractures.
We denote by $\kappa_i$ the permeability of the $i$-th fracture.
$D_m$ is a two-dimensional domain and $D_{f,i}$ is a
one-dimensional domain.
The system is written in a finite{-}element discretization.
We introduce the concepts of fine and coarse grids.
Let $\mathcal{T}^H$ be a coarse-grid partition (computational grid)
of the computational domain
$D$ into finite elements (triangles, quadrilaterals, tetrahedra, etc.).
We assume that each coarse element is partitioned into
a connected union of fine grid blocks.
The fine-grid partition will be denoted by
$\mathcal{T}^h$, and is by definition a refinement of the coarse grid
$\mathcal{T}^H$.
We use $\{x_i\}_{i=1}^{N}$ (where $N$ denotes the number of coarse nodes) to denote the vertices of
the coarse mesh $\mathcal{T}^H$ and define the neighborhood of the node $x_i$ by
\begin{equation} \label{neighborhood}
\omega_i=\bigcup\{ K_j\in\mathcal{T}^H; ~~~ x_i\in \overline{K}_j\}.
\end{equation}
See Figure \ref{schematic} for illustration.

The bilinear form for the resulting system is
\begin{equation}
\label{eq:main_eq1}
\begin{split}
\int_{D_m}  c_m {\partial u_{h} \over \partial t} v_h \, {dx}  &+
 \sum\limits_i
\int_{D_{f,i}} c_{f,i}  {\partial u_{h} \over \partial t} v_h \, {dx} \\
&+\int_{D_m}   \kappa_m \nabla u_{h} \cdot \nabla v_h \, {dx} +
 \sum\limits_i
\int_{D_{ f,i}} \kappa_{f,i}\nabla_f u_{h}\cdot \nabla_f v_h \, {dx}  = \int_D q v_h \, {dx},
\end{split}
\end{equation}
where $v_h$ is the fine-grid finite element function, $\nabla_f$ is
the derivative along the fracture lines,
$c_m$ and $\kappa_m$ porosity and permeability in the matrix,
$c_{f,i}$ and $\kappa_{f,i}$ porosity and permeability in the fractures,
and
$i=1,...,N$.
The fracture permeability
and porosity include the aperature information $d_i$.
We remind that
in our setup, we assume that a fine grid resolves some set of fractures (very detailed), while each fine grid can contain many small fractures (see Figure \ref{schematic}),
i.e., multi-continua.

\begin{figure}[!ht]
  \centering
  \includegraphics[width=0.6 \textwidth]{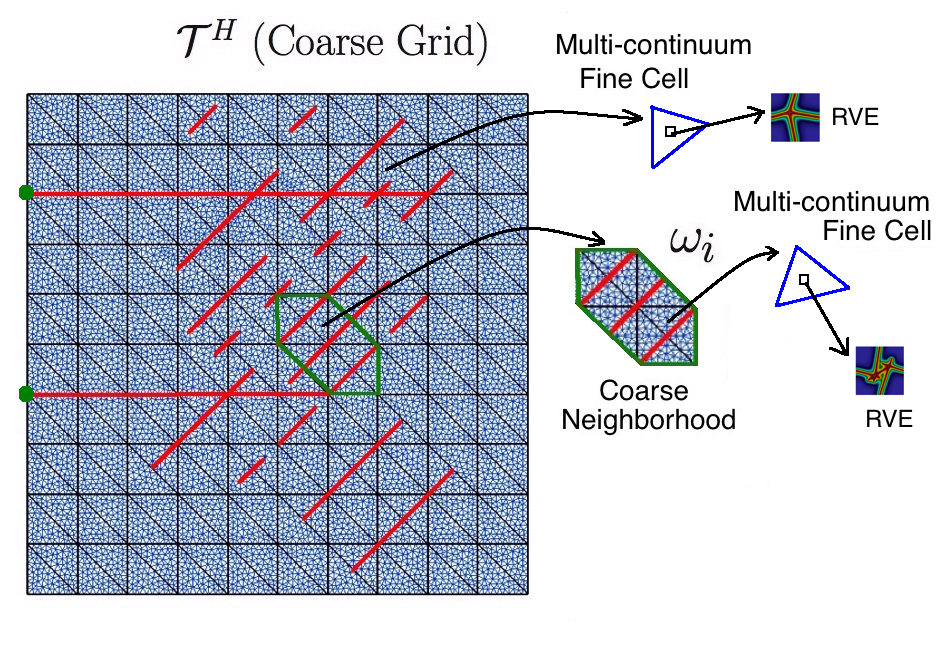}
  \caption{Illustration of a coarse neighborhood and coarse element.}
  \label{schematic}
\end{figure}

\subsection{Multi-continuum approach. A brief summary.}

The multi-continuum approach is an average model,
which is solved on a coarse grid.
We denote the solution for $i${-}th continuum by $u_i$ and
assume that each continuum
interacts with every other, for {the} sake of generality.
Then, we can write the resulting system as
\begin{equation}
\label{eq:multi_cont}
{\mathcal C}_{ii} {\partial u_i \over \partial t}=\text{div}(\kappa_i \nabla u_i) + Q_i(u_1,...,u_N),
\end{equation}
where $Q_i(u_1,...,u_N)$ is an exchange term, which can contain both space
and time derivates of $u_i$'s {\cite{arbogast1990derivation,barenblatt1960basic,kazemi1976numerical,pruess1982fluid,warren1963behavior,wu1988multiple}}.
 In a special case when each continuum only interacts with the
background, if the background is $u_1$, then
\[
{\mathcal C}_{ii}{\partial u_i \over \partial t}=\text{div}(\kappa_i \nabla u_i) + Q_i(u_1,u_i) + q,
\]
where $q$ is the source term.
We write this equation as
\begin{equation}
\label{eq:main11}
\begin{split}
\int_{D} {\mathcal C}_{ii}{\partial u_i \over \partial t} v \, {dx}
+\int_{D}   \kappa_i \nabla u_i \cdot \nabla v \, {dx}  {-}
\int_{D}   Q_i  v \, {dx}
= \int_D q v \, {dx},
\end{split}
\end{equation}
where $u_i$ is solved on a coarse grid using standard basis functions
and $v$ is a standard basis function.

\subsection{A brief overview of the GMsFEM.}
\label{sec:overview}

We discuss the use of the GMsFEM on a
coarse grid.
The GMsFEM uses the coarse grid $\mathcal{T}^H$ and  constructs
a local reduced-order model for each coarse block, by constructing snapshot
solutions and extracting basis functions. Snapshots are constructed by
solving local problems subject to some boundary conditions.
Below, we briefly discuss the snapshot calculations and
the basis computations.

\subsubsection{Snapshots and multiscale basis.}

We briefly describe the construction of the snapshot space
$V_{\text{snap}}^{\omega_i}$. We refer to \cite{chung2016adaptive, egh12}
for further discussions. The snapshot space consists
of local solutions.
Harmonic functions can be used  to construct a snapshot space. We define
$\delta_l^h(x)$, where
$\delta_l^h(x)=\delta_{l,j},\,\forall l,j\in \textsl{J}_{h}(\omega_i)$,
where $\textsl{J}_{h}(\omega_i)$ denotes the fine-grid boundary node on $\partial\omega_i$ and solve the local problems with $\delta_l^h(x)$
as boundary conditions.
More precisely,
given a fine-scale piecewise linear function defined on
$\partial\omega$ ($\omega$ is a coarse block and we omit the index $i$),
we define $\psi_{l}^{\omega, \text{snap}}$ by {the} following variational problem
\begin{equation}
\label{harmonic_ex}
a(\psi_{l}^{\omega, \text{snap}} , v) =
\int_{\omega}   \kappa_m \nabla \psi_{l}^{\omega, \text{snap}} \cdot \nabla v_h \, dx +  \sum\limits_j
\int_{D_{f,j}\cap \omega} \kappa_{f,j}\nabla_f \psi_{l}^{\omega, \text{snap}} \cdot \nabla_f v_h \, dx
=0 \quad \text{in } \, \omega
\end{equation}
and $\psi_{l}^{\omega, \text{snap}}=\delta_l^h(x)$ on $\partial\omega$.
Note that the source
is also placed on fracture boundaries.
The snapshot space is defined as
$$
V_{\text{snap}} = \text{span}\{ \psi_{l}^{ \text{snap}}:~~~ 1\leq l \leq L_i \},
$$
where $L_i$ is the number of functions in the snapshot space in  $\omega$
(a generic coarse block). We also denote
$$
R_{\text{snap}} = \left[ \psi_{1}^{\text{snap}}, \ldots, \psi_{L_i}^{\text{snap}} \right].
$$
We note that {the randomized boundary conditions \cite{randomized2014}
can be used} to reduce the computational cost. In particular, we solve local problem{s}
subject to the boundary condition
\[
\psi_{l}^{ \text{snap}}=r_l,
\]
where $r_l$ takes independent random values at every grid block
in an oversampled region $\omega^+$, $\omega\subset \omega^+$ (see \cite{randomized2014} for details). In this way, we can compute
only $n+4$
snapshots for $n$ offline basis vectors.

To construct the offline space, $V_{\text{off}}^\omega$, {the} local spectral problem is solved in the snapshot space \cite{egw10}.
More precisely,
\begin{equation}
\label{offeig}
A^{\text{off}} \Psi_l^{\text{off}} = \lambda_l^{\text{off}} S^{\text{off}} \Psi_l^{\text{off}},
\end{equation}
where
\[
A^{\text{off}} = [a_{mn}^{\text{off}}] =
\int_{\omega}   \kappa_m \,  \nabla \psi_m^{\text{snap}} \cdot \nabla \psi_n^{\text{snap}} \, dx +
 \sum\limits_j
\int_{D_{ f,j}\cap \omega} \kappa_{f,j} \,  \nabla_f \psi_m^{\text{snap}} \cdot \nabla_f \psi_n^{\text{snap}} \, dx
\] \[
S^{\text{off}} = [s_{mn}^{\text{off}}] =
\int_{\omega}   \kappa_m\,  \psi_m^{\text{snap}} \psi_n^{\text{snap}} \, dx +
 \sum\limits_j
\int_{D_{ f,j}\cap \omega} \kappa_{f,j} \,  \psi_m^{\text{snap}} \psi_n^{\text{snap}} \, dx.
\]
To compute the offline space, we choose  $M^{\omega}_{\text{off}}$  smallest eigenvalues  and form
$\psi_m^{\text{off}} = \sum_{l=1}^{L_i} \Psi_{ml}^{\text{off}} \psi_l^{\text{snap}}$ for $m=1,\ldots, M^{\omega}_{\text{off}}$.
Furthermore,
the partition of unity functions $\chi_i$ (taken to be
linear basis functions supported in $\omega_i$) is multiplied
by the eigenfunctions in the offline space $V_{\text{off}}^{\omega_i}$
to construct the resulting basis functions
\begin{equation} \label{cgbasis}
\psi_{i,j} = \chi_i \psi_j^{\omega_i, \text{off}} \quad \text{for} \, \, \,
1 \leq i \leq N \, \, \,  \text{and} \, \, \, 1 \leq j \leq M_{\text{off}}^{\omega_i}.
\end{equation}
Here $M_{\text{off}}^{\omega_i}$ denotes the number of offline eigenvectors that are selected for each coarse node $i$. With the partition of unity functions,
we obtain conforming basis functions in the space
\begin{equation} \label{cgspace}
V_{\text{off}}  = \text{span} \{ \psi_{i,j} : \,  \, 1 \leq i \leq N \, \, \,  \text{and} \, \, \, 1 \leq j \leq M_{\text{off}}^{\omega_i}  \}.
\end{equation}
We can write $V_{\text{off}} = \text{span} \{ \psi_{i} \}_{i=1}^{N_c}$, where $N_c =\sum_{i=1}^{N}M_{\text{off}}^{\omega_{i}}$ (here, we use a single index)
and define
\[
R^T = \left[ \psi_1 , \ldots, \psi_{N_c} \right],
\]
where $\psi_i$ are nodal values of each basis function defined on the fine grid.

{We remark that there are other} discretizations, such as discontinuous Galerkin methods,
hybridized Galerkin methods, or other methods. Multiscale basis
functions can be constructed following a general framework
\cite{chung2016adaptive}. The use of discontinuous basis functions
coupled within DG can be an attractive approach for these applications
and we will study it in {the} future.

\subsection{A numerical example demonstrating fracture networks and associated eigenvalues}

Next, we discuss some properties of multiscale basis functions, which
show that the GMsFEM basis functions can identify the fracture
networks in a general case. Further, we present some simplified
basis computations, when the fracture networks have simplistic geometries.
We note that each basis function
represents a connected fracture network. To show this,
we depict an example in Figure \ref{sol-bst}
with several fractures. In the figure,
we also show the eigenvalues. {It can be observed} that
there are three very small eigenvalues and the fourth
one is large. The eigenvalue distribution shows
that there are three fracture networks. Our construction
can detect the fracture networks when fractures have a complex spatial
distribution.
Moreover, multiscale basis functions can capture the interaction
between the fracture and the background media.

\begin{figure}[!ht]
  \centering
    \includegraphics[width=0.2 \textwidth]{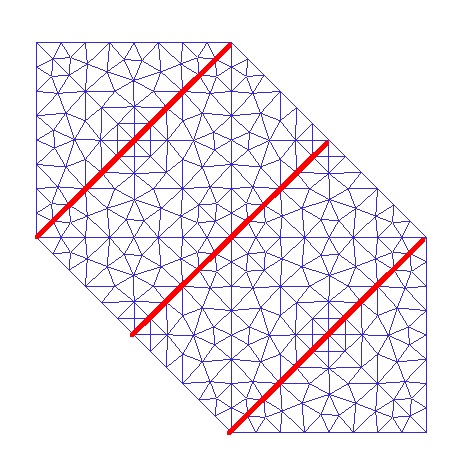}
    \includegraphics[width=0.75 \textwidth]{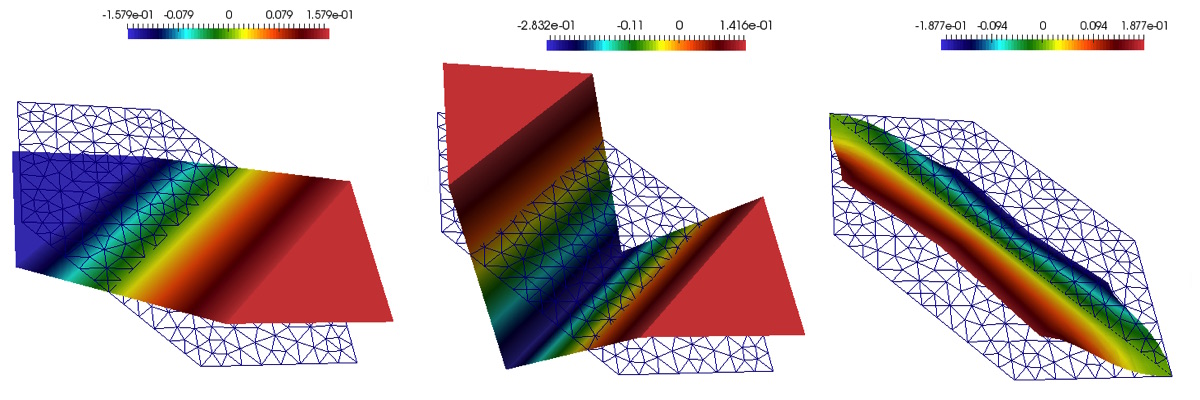}
    \caption{GMsFEM basis functions in a domain $\omega$.
      The eigenvalues are
    $\lambda_1 = 1.26 \cdot 10^{-14}$,
    $\lambda_2 = 2.3 \cdot 10^{-7}$,
    $\lambda_3 = 7.0 \cdot 10^{-7}$,
    $\lambda_4 = 0.16$. }
  \label{sol-bst}
\end{figure}

\subsection{Simplified basis functions.}

For simple cases, {simplified basis functions can be constructed}.
For these basis functions, we can choose constants
within fracture networks and solve  local problems.
In this way, we can avoid a general procedure.
Our main approach, which we will test, is the following.
For each $\omega$, we define the fracture networks
$\Gamma_1^\omega,...,\Gamma_M^\omega$ (see Figure \ref{schematic_simplified}).
Each fracture network
intersect{s with} the boundary of $\omega$ at the points $B_{i}^{\Gamma_j^{\omega}}$.
Then, the multiscale basis functions are defined as
\[
L(\phi_m)=0,
\]
\[
\phi_m(B_{i}^{\Gamma_j^{\omega}})=\delta_{mj}.
 \]
Here, $L$ corresponds to the local solution operator (\ref{harmonic_ex}).
These basis functions are multipled by the partition of unity functions.
The basis functions are plotted in Figure \ref{sol-bsm}.

\begin{figure}[!ht]
  \centering
    \includegraphics[width=0.2 \textwidth]{om81}
    \includegraphics[width=0.75 \textwidth]{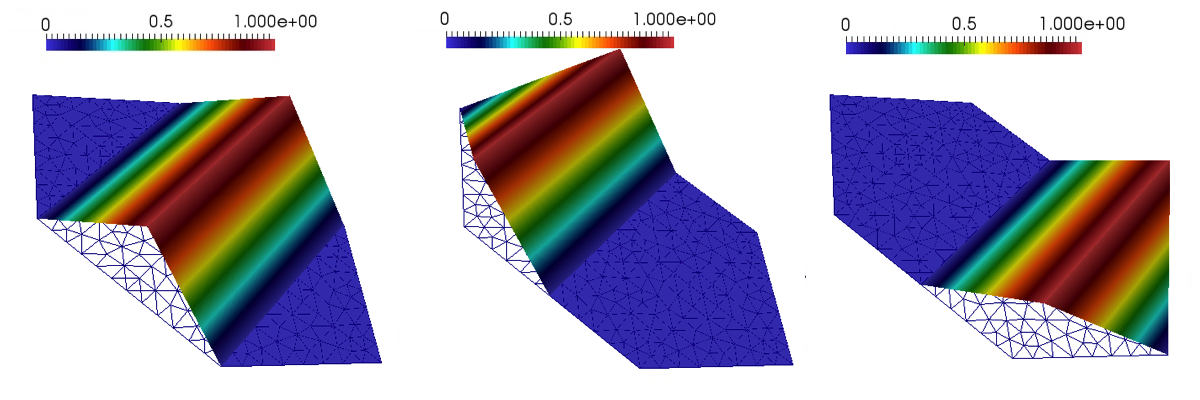}
    \caption{Simplified basis functions in a domain $\omega$. }
  \label{sol-bsm}
\end{figure}

\begin{figure}[!ht]
  \centering
  \includegraphics[width=0.6 \textwidth]{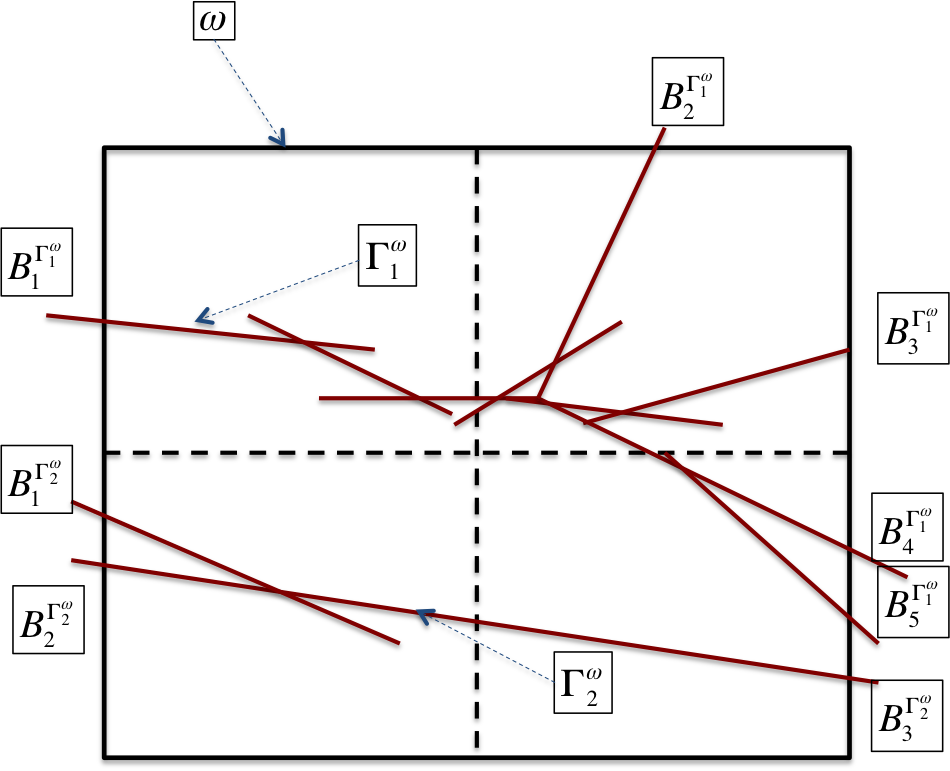}
  \caption{Illustration of coarse neighborhood and simplified basis functions{.} }
  \label{schematic_simplified}
\end{figure}

\begin{remark}
  {There are other possible approaches that can be considered}. For example,
  the following approach can be an alternative.
    We denote each rectangle $K\subset \omega$
 and denote internal edges by $\partial \omega^I$.
 We denote the fractures by $\Gamma_1^K,...,\Gamma_M^K$ in $K$
 and the boundary nodes
 \[
B_i=B_{i}^{\Gamma_j^K}.
\]
Then, the multiscale basis functions are defined as
\[
L(\phi_m)=0,
\]\[
\phi_m(B_{i})=\delta_{mi}.
 \]
\end{remark}

 \subsection{Relating basis to fractures}


If we denote the basis function for {the} $i$-th
network by $\phi_i^{\omega_j}$, as described above, then
\[
u_i = \sum_j c_i^{\omega_j}\phi_i^{\omega_j}.
\]
Note that
\[
u = \sum_i u_i.
\]
In this case, the coarse-grid equation
obtained by the GMsFEM for the basis representing {the}
$i$-th fracture network can be written
as
\[
\int_D c {\partial u_i \over \partial t}\phi_i^{\omega_l} \, {dx} +
\int_D \kappa \nabla u_i \cdot \nabla \phi_i^{\omega_l} \, {dx} =
\int_D q \phi_i^{\omega_l}  \, {dx} -
\sum_{s, s\not = i}
\int_D c {\partial u_s \over \partial t}\phi_i^{\omega_l} \, {dx}
-
\sum_{s, s\not = i}  \int_D \kappa \nabla u_s \cdot \nabla \phi_i^{\omega_l} \, {dx},
\]
$l=1,...$
The last two terms represent the interaction of {the} $i$-th continu{um} with
the other continua. In a special case when
 the interaction is only with the background,
this implies that the support of $\phi_i^{\omega_l}$ and
$\phi_j^{\omega_m}$ is empty unless $j=1$ or $j=i$. In this case, the
equation reduces to
\[
\int_D c {\partial u_i \over \partial t}\phi_i^{\omega_l} \, {dx}+
\int_D \kappa \nabla u_i \cdot \nabla \phi_i^{\omega_l}  \, {dx}=
\int_D q \phi_i^{\omega_l} \, {dx} -
\int_D c {\partial u_1 \over \partial t}\phi_i^{\omega_l} \, {dx}
-
\int_D \kappa \nabla u_1 \cdot \nabla \phi_i^{\omega_l} \, {dx}.
\]

\subsection{A numerical example}

In this example,
we take $D = [0, 60]^2$ and solve
\[
  c {\partial u\over\partial t} - {\text{div}}(\kappa \nabla u) = 0, \quad x \in D
\]
by resolving the fractures with {an} embedded fracture model
on the fine grid (see e.g., \cite{akkutlu2015multiscale}).
The model for $\kappa$ is shown in Figure \ref{schematic}.
We choose  initial conditions
$u =u_f = 1$ and,
as the boundary conditions, we set $u = 0$ {at the} two points $(0, 24)$ and $(0, 48)$ and on other boundaries we use zero Neuman boundary conditions. Here, $T_{max} = 300$ is the final time. We set $c_m = 0.1$, $\kappa_m = 10^{-2}$ for {the} matrix coefficients and $c_f = 0.01$, $\kappa_f = 10^{4}$ for {the} fracture.

{
We will compare the results in the weighted $L^a_2(u)$ norm and weighted $H^a_1(u)$ semi-norm computed as
\[
||e_u||_{L_2} = ||u - u_h||_{L_2} / || u_h ||_{L_2},
\quad
|e_u|_{H_1} = |u - u_h|_{H_1} / | u_h |_{H_1},
\]
where $||u||^2_{L_2} = \int_{\Omega} k \, u^2 \, dx$, $|u|^2_{H_1} = \int_{\Omega} (k \, \nabla u, \nabla u) \, dx$, $u_h$ and $u$ are the fine-scale and coarse-scale (multiscale) solutions.
In the simulation results, we use $\mathcal{M}$ to denote the number of basis functions per coarse element for $u$ and $DOF$ is the number of degrees of freedom.
}

In Figure \ref{sol-s}, we show solutions at the final time $T_{max}$.
In Table \ref{tab:s}, we present relative errors for GMsFEM and simplified basis functions.
The top portion of the table (``Standard GMsFEM'') uses multiscale
basis functions constructed from the spectral problems and take {the}
equal amount of basis functions in each coarse region. Here,
$\mathcal{M}$ refers to the number of basis functions per node.
As we observe that if we take $4$ basis functions per node, the error
is below 5\%. In the second portion of the table, we
show the results if basis functions are
selected based on small eigenvalues.
$M_{\lambda}$ refers to the case when we take only very small eigenvalues
that represent the fracture networks. In this case, the error is small.
$M_{\lambda}-1$ and $M_{\lambda}+1$ refer to the cases when we take one less or
one more basis functions in each node. In the bottom portion, we use simplified
basis functions. As we observe that the simplified basis captures the networks and provide a small error.


\begin{figure}[!ht]
  \centering
  \includegraphics[width=0.32 \textwidth]{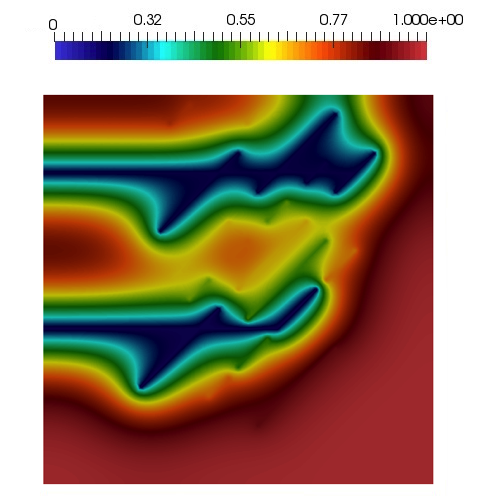}
  \includegraphics[width=0.32 \textwidth]{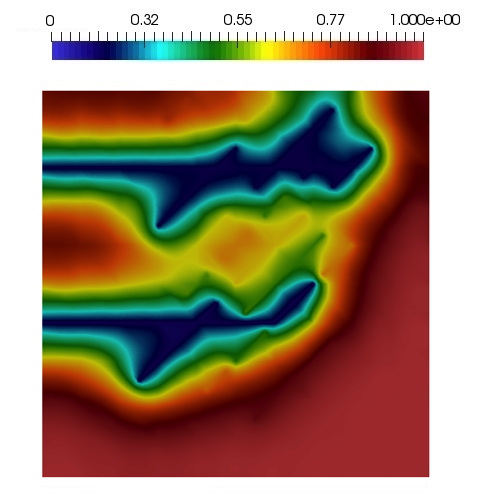}
  \includegraphics[width=0.32 \textwidth]{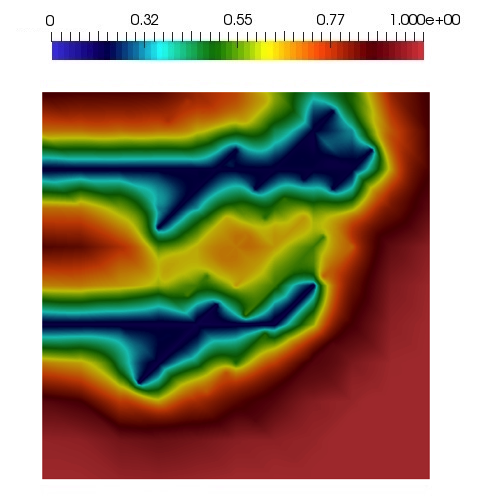}
  \caption{Single-continuum background results.  Left: Fine-scale solution $DOF = 8917$. Middle: Coarse-scale solution $DOF = 396$ using GMsFEM. Right: Coarse-scale solution $DOF = 268$ using simplified basis functions.}
  \label{sol-s}
\end{figure}

\begin{table}[!ht]
\begin{center}
\begin{tabular}{|c|c|c|c|}
\hline
$\mathcal{M}$  & dim($V_{\text{off}}$) & $L^a_2(u)$ & $H^a_1(u)$ \\
\hline
\multicolumn{4}{|c|}{Standard GMsFEM}  \\ \hline
1		&	 121		& 49.067	& 79.885 \\
2		& 	 242		& 9.116	& 36.499  \\
3		&	 363		& 2.325	& 8.099 \\
4		& 	 484		& 1.413	& 3.919\\
5		&	 605		& 0.883	& 2.658 \\
6		&	 726		& 0.708	& 1.795 \\
8		&	 968		& 0.253	& 0.348 \\
16    &  1936	& 0.095	& 0.089 \\
\hline
\multicolumn{4}{|c|}{GMsFEM by $\lambda$}  \\ \hline
$M_{\lambda}-1$    &  184		&	 13.554	& 45.342 \\
$M_{\lambda}$    &  275		&	 2.651	& 11.377 \\
$M_{\lambda}+1$     &  396		&	 1.582	& 4.414 \\
\hline
\multicolumn{4}{|c|}{Simplified basis functions}  \\ \hline
all    			&  268		&	1.850	& 3.799 \\
\hline
\end{tabular}
\end{center}
\caption{Single-continuum background. Numerical results of relative errors (\%) at the final simulation time. $DOF_f = 8917$.}
\label{tab:s}
\end{table}

\section{The coupled GMsFEM and multi-continuum}
\label{sec:sec2}

In this section, we discuss a combined GMsFEM and multi-continuum method.
We assume that some fractures are resolved on the fine grid, while
other fractures are represented using {a} multi-continuum approach
at the fine-grid
level. As a result, we deal with a system of equations with reaction
tensors.
We note that each continuum interacts with the resolved fractures and,
also, they interact among themselves.
We discuss coupled and un-coupled basis constructions.
The couple{d} basis functions are important for some
flow scenarios as we discuss. The analysis of the method
is given in Appendix \ref{appendix}.

We will consider two cases. In the first case, we simply
use some values for transfer coefficients $Q$ and
in the second case, we compute these transfer coefficients
from RVE simulations. In both cases, we assume that each
continuum is connected to the fracture.
Thus, the fractures are added to each continuum equation with an appropriate  weight
$\gamma_i$, which represents the amount of the
fluid passed to the fracture
network from the $i$-th continu{um}. We can assume $\sum_i \gamma_i=1$.
The resulting equations have the following
variational form
\begin{equation}
\label{eq:discr1}
\begin{split}
\int_{D_m} &c_{m,s}   {\partial u_s \over \partial t} v \, {dx}  +
 \sum\limits_i
\int_{D_{f,i}} c_{i,s} {\partial u_s \over \partial t} v \, {dx} \\
&+\int_{D_m}   \kappa_s \nabla u_s \cdot \nabla v \, {dx} +
 \sum\limits_i
\int_{D_{ f,i}} \kappa_{i,s}\nabla u_s\cdot \nabla v \, {dx}  = \int_D Q_s v \, {dx},
\end{split}
\end{equation}
$s=1,...,N$.
Here, $\kappa_{i,s}$ is the fracture permeability that takes into account
the interaction of {the} $s$-th continuum with the resolved fracture network
and {$c_{i,s}$} is the mass exchange term that take into account
the interaction between the fracture and the $s$-th continuum.
Note that $c_{i,s}$,  $\kappa_{i,s}$, and $Q_s$ depend on $\gamma_s$.
This is a coupled system of differential equations with multiscale
high-contrast coefficients. The coupling is done via
the right hand side and, thus,
multiscale basis functions can be constructed
for separately for each equation using the high-contrast permeabilities
or jointly.

In our numerical simulations, we will consider two approaches for
constructing multiscale spaces as described in Appendix \ref{appendix}.
In the first approach (called un-coupled), multiscale basis
functions will be constructed for each continuum separately by
considering only the permeability $\kappa_i$ and ignoring the
transfer functions. This is the same as using single-phase
flow basis functions for each continuum and
follows the GMsFEM approach discussed
above.
In the second approach, the multiscale basis functions will
be constructed by solving {a} coupled problem for snapshot spaces
and performing a spectral decomposition as discussed in Appendix \ref{appendix}.
{The resulting GMsFEM procedure is the same as the one presented in Section \ref{sec:overview},
except that the construction of the snapshot functions is replaced by the coupled approach discussed above.}
Note that a different spectral problem is used for a coupled
basis construction.


We present numerical results.
We consider the model shown in Figure \ref{schematic}.
We consider {a} dual porosity system for
un-resolved fractures ($u_f$) and matrix flow ($u_m$).
For {the}  un-resolved fractures parameters, we set $\kappa_f = 10^{-3}$ and
$c_f = 0.1$. For {the} matrix parameters, we use $\kappa_m = 10^{-7}$ and
$c_m = 0.01$. Using DFN {(discrete fracture network)}, we implement {a} resolved fracture network
($u_F$), which interacts with both {the} un-resolved fracture system
($80\%$) and {the} matrix system ($20\%$). That is both matrix and un-resolved
fracture system communicate with the resolved fractures.
We set $\kappa_F = 10^{3}$, $c_F = 0.1$. Here, we use
the transfer parameter $Q = 250 \cdot \kappa_m$ and $T_{max} = 5000$.
In Figure \ref{sol-dp}, we show solutions at the final time.
In Table \ref{tab:dp-cp},
we present relative errors for GMsFEM and simplified basis functions.
In this table, we present the results when {the} basis {is} computed in a coupled way and
separately for each continuum using {the} flow equation (without transfer functions).
From this table, we observe that the GMsFEM using coupled basis
functions provides better accuracy compared to that computed with
un-coupled basis functions. Moreover, we observe
that when choosing $6$ basis functions per coarse node, we can obtain an excellent
result using the GMsFEM.
{We have also tested the
GMsFEM with simplified basis functions. The results are similar to those obtained from above.
In particular, we observe} a similar accuracy if we choose only basis functions
corresponding to very small eigenvalues. We observe large errors if
we do not choose eigenvectors corresponding to very small
eigenvalues.

\begin{figure}[!ht]
  \centering
  \includegraphics[width=0.49 \textwidth]{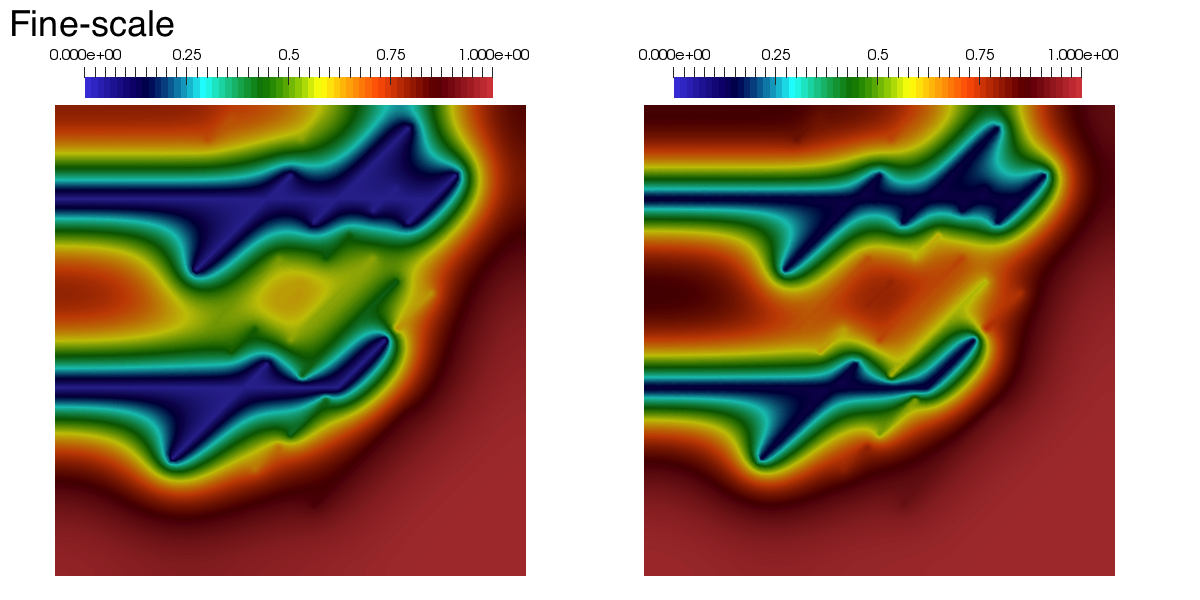}
  \includegraphics[width=0.49 \textwidth]{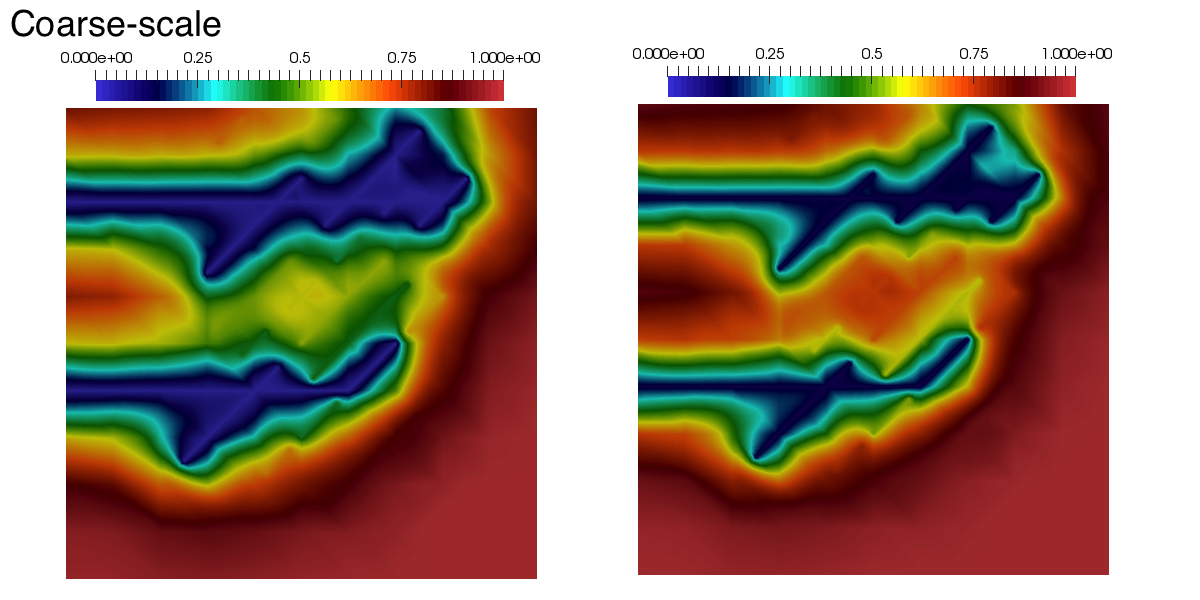}
  \caption{Dual-continuum background.
  Left: Fine-scale solution $DOF = 17834$.
  Right: Coarse-scale solution $DOF = 536$ using simplified basis functions.}
  \label{sol-dp}
\end{figure}

\begin{table}[!ht]
\begin{center}
\begin{tabular}{|c|c|c|c|c|c|c|}
\hline
$\mathcal{M}$  & dim($V_{\text{off}}$)
& $L^a_2(c_1)$ & $H^a_1(c_1)$
& $L^a_2(c_2)$ & $H^a_1(c_2)$ & $H^Q(c_1, c_2)$ \\
\hline
\multicolumn{7}{|c|}{Standard GMsFEM (un-coupled)}  \\ \hline
6		&	 1452	& 0.846	& 1.966	&  	0.846	& 15.938	& 4.519 \\
8		&	 1936	& 0.297	& 0.379	&  	0.297	& 15.835	& 4.093 \\
12	&	 2904	& 0.148	& 0.188	&  0.148		& 14.213	& 3.670 \\
16    &  3872	& 0.106	& 0.087	& 	0.106		& 11.769	& 3.041 \\
\hline
\multicolumn{7}{|c|}{Simplified basis functions (un-coupled)}  \\ \hline
all    	&  536		&	2.386	& 3.831	& 2.386	& 21.765	& 6.713  \\
\hline
\multicolumn{7}{|c|}{Standard GMsFEM (coupled)}  \\ \hline
6		&	 1452	& 1.919	& 2.767	&  	1.919	& 9.919	& 3.588 \\
8		&	 1936	& 1.052	& 1.154	&  	1.052	& 2.840	& 1.335 \\
12	&	 2904	& 0.354	& 0.550	& 	 	0.354	& 1.325	& 0.628 \\
16    &  3872	&0.124		& 0.102	&  	0.124	& 0.643	& 0.194 \\
\hline
\multicolumn{7}{|c|}{Simplified basis functions  (coupled)}  \\ \hline
all    	&  830		&	2.056	& 3.439	&  	2.056	& 6.234	& 3.690 \\
\hline
\end{tabular}
\end{center}
\caption{Dual-continuum background. Numerical results of relative errors (\%) at the final simulation time. $DOF_f = 17834$. $Q = 250 \cdot \kappa_m$.}
\label{tab:dp-cp}
\end{table}

\subsection{RVE-based multi-continuum computations}

{A r}epresentative volume element can be used
to compute the parameters in multi-continuum equations.
We follow a known procedure (see e.g., \cite{karimi2016general}).
 In this approach, we compute the transfer parameters based on RVE
 simulations.
 We consider a case of two-continua at the microscale and compute
the transfer function based on RVE simulations.
For this reason, we solve the local problem with DFN
(corresponding to (\ref{harmonic_ex}))
\[
c  {\partial \xi \over\partial t} - {\text{div}} (\kappa \nabla \xi)=0\ \text{{in the} RVE}
\]
and impose $\xi=1$ at the fracture nodes. One can also use a source term
in the fracture or the local eigenvalue problems see \cite{karimi2016general}.
It is assumed that zero Neumann boundary conditions are imposed on the rest of the boundaries.
Then, the transfer coefficient is defined as
\[
  Q(t)=F_{\text{frac}}(t)/(1-\langle \xi\rangle_{\text{matrix}}(t)).
\]
$Q(t)$ will quickly reach an asymptote, which is used as a transfer coefficient.
Here $\langle \xi\rangle_{\text{matrix}}$ is the volume average over the matrix region.


We present results for {a} dual-continuum coupled with the GMsFEM.
We set $c_f = c_m = 0.1$, $\kappa_f = 10^{-3}$, $\kappa_m = 10^{-7}$,
 and $c_F = 0.01$, $\kappa_F = 10^{3}$ for {the} fracture. Using DFN,
we implement {the} resolved fracture network with $80\%$ in
$c_f$ and $20\%$ to $c_m$ {as well as} $T_{max} = 5000$.
Here, we use {the} transfer function{s} $Q_1 = 500 \cdot \kappa_m$ for
$y < 5L_y/10 $, $Q_2 = 920 \cdot \kappa_m$ for $y >  7L_y/10$,
  and linear{ize} $Q$ in between them, i.e., $y \in  5L_y/10, 7L_y/10)$.
The values of $Q$ {are} computed using local RVE simulations.
In particular, we set the pressure to be one and compute
$Q$ as the flux over the pressure difference in the fracture
and the average pressure in the matrix (see \cite{karimi2016general}).

{In Figure \ref{sol-dpQ-cp}{,} we plot both the fine-scale and the coarse-scale
solutions at the final time, and
in Table \ref{tab:dpQ-cp}, we report the relative errors for GMsFEM and simplified basis functions.
The numerical results are obtained using the parameters as follows:
$c_f = c_m = 0.1$, $\kappa_f = 10^{-3}$, $\kappa_m = 10^{-7}$ and $c_F = 0.01$,
$\kappa_F = 10^{3}$ for {the} fracture.
The resolved fracture network is implemented using DFN
with $80\%$ in $c_f$ and $20\%$ to $c_m$.
In addition, we take the transfer function
$Q = 250 \cdot \kappa_m$ and $T_{max} = 5000$.
The final time solution plots are shown Figure \ref{sol-dpQ-cp},
and the corresponding relative errors are reported in Table \ref{tab:dpQ-cp},
where the results are shown when the basis are
computed in a coupled way and
separately for each continuum using flow equation (without transfer functions).
Based on these results, we conclude that the GMsFEM using coupled basis
functions provides better accuracy compared to that computed with
un-coupled basis functions. Furthermore, we observe
that when choosing $6$ basis functions per coarse node, we can obtain an excellent
result using the GMsFEM.
On the other hand, we tested the performance of our method using simplified basis functions,
and observed a similar accuracy.
Finally,
we observe large errors if
we do not choose eigenvectors corresponding to very small
eigenvalues.}

\begin{figure}[!ht]
  \centering
  \includegraphics[width=0.49 \textwidth]{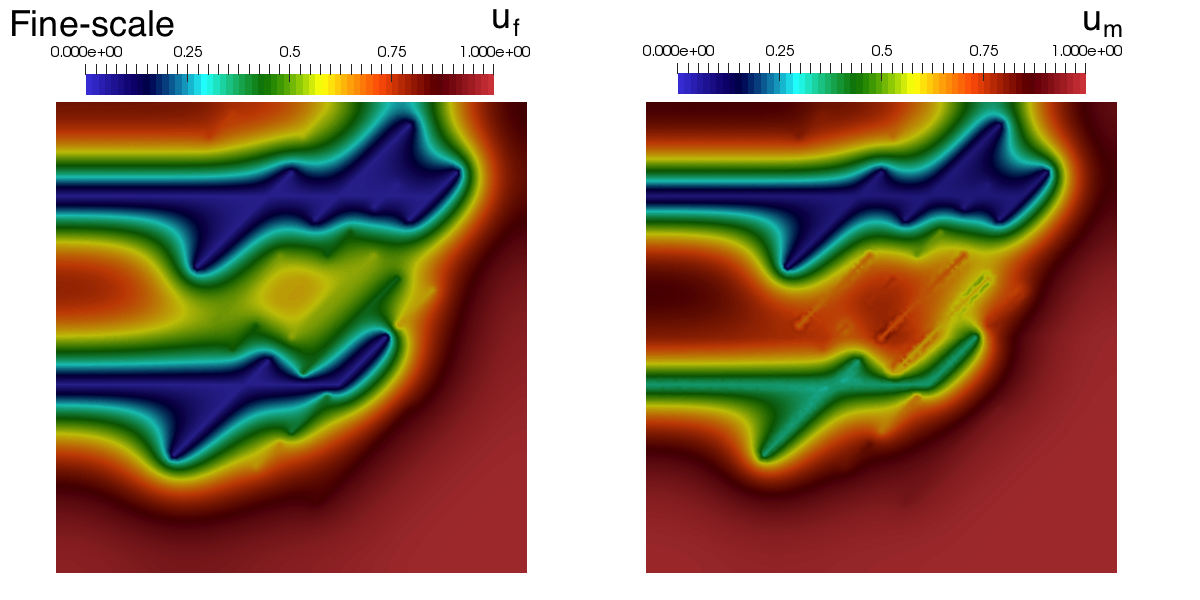}
  \includegraphics[width=0.49 \textwidth]{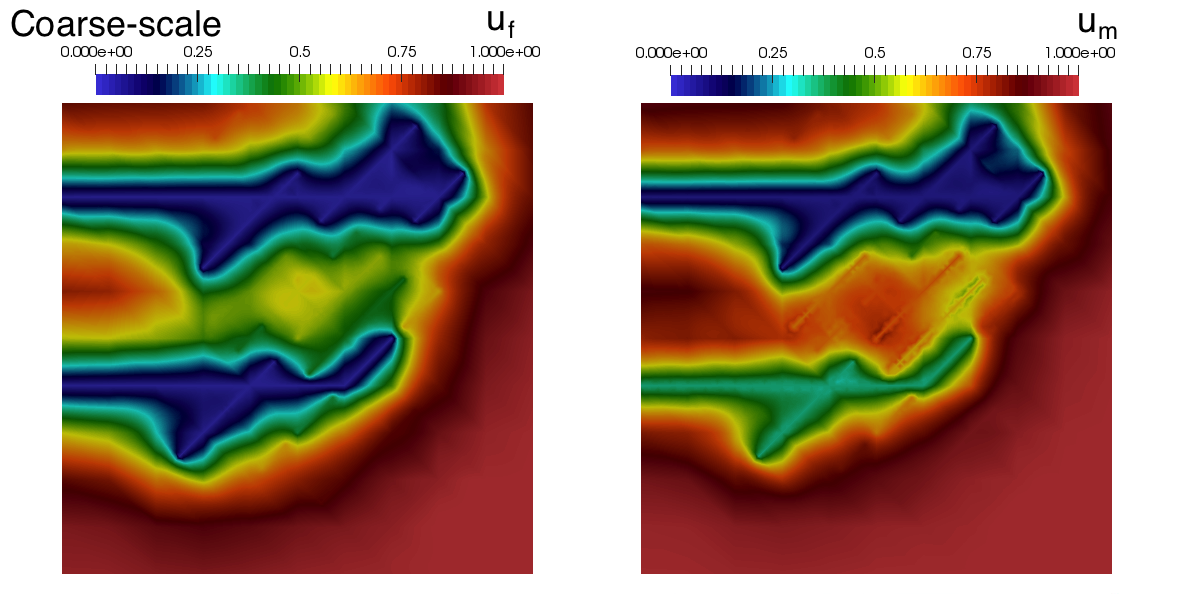}
  \caption{Dual-continuum background with $Q= Q(x)$.
  Left: Fine{-}scale solution $DOF = 17834$.
  Right: Coarse-scale solution $DOF = 536$ using simplified basis functions. $Q_1 = 250 \cdot \kappa_m$ and $Q_2 = 920 \cdot \kappa_m${.} }
  \label{sol-dpQ-cp}
\end{figure}

\begin{table}[!ht]
\begin{center}
\begin{tabular}{|c|c|c|c|c|c|c|}
\hline
$\mathcal{M}$  & dim($V_{\text{off}}$)
& $L^a_2(c_1)$ & $H^a_1(c_1)$
& $L^a_2(c_2)$ & $H^a_1(c_2)$ & $H^Q(c_1, c_2)$ \\
\hline
\multicolumn{7}{|c|}{Standard GmsFEM (un-coupled)}  \\ \hline
6 		& 1452 	& 0.837	& 1.925	&  0.837	& 30.542	& 13.869 \\
8 		& 1936 	& 0.293	& 0.358	&  0.293	& 26.427	& 11.912 \\
12 	& 2904		& 0.148	& 0.193	&  0.148	& 23.165	& 10.440  \\
16 	& 3872		& 0.108	& 0.089	&  0.108	& 17.945	& 8.086  \\
\hline
\multicolumn{7}{|c|}{Simplified basis functions (un-coupled)}  \\ \hline
all 	& 536 	& 2.343	& 3.870	& 2.343	& 36.649	& 16.872 \\
\hline
\multicolumn{7}{|c|}{Standard GmsFEM (coupled)}  \\ \hline
6		& 1452 	& 1.944	& 2.584	&  	1.944	& 6.942	& 3.934  \\
8		& 1936 	& 1.070	& 1.200	&  	1.070	& 2.197	& 1.452 \\
12	& 2904 	& 0.359	& 0.544	&  	0.359	& 0.788	& 0.606 \\
16	& 3872 	& 0.129	& 0.105	&  	0.129	& 0.375	& 0.193 \\
\hline
\multicolumn{7}{|c|}{Simplified basis functions  (coupled)}  \\ \hline
all 	& 830		&	 2.105	& 3.399	& 	 	2.105	& 4.122	& 3.557  \\
\hline
\end{tabular}
\end{center}
\caption{Dual-continuum background. Numerical results of relative
  errors (\%) at the final simulation time. $DOF_f = 17834$. $Q_1 = 250 \cdot \kappa_m$ and $Q_2 = 920 \cdot \kappa_m$.}
\label{tab:dpQ-cp}
\end{table}

\begin{remark}

{We remark that the RVE can be used} to approximate the effective properties.
To show this example, we assume that in each fine-grid, the multi-continua
can be resolved. In this case, we construct multiscale basis functions
via local spectral decomposition in the form
\[
\phi_{i,\text{fine}}^{\omega_j}=\chi^{\omega_i}_{\text{fine}} \psi_{j,\text{fine}}.
\]
As we discussed above, this equation is a multi-continua model obtained via the GMsFEM and can be related to multi-continua (\ref{eq:multi_cont}) by constructing {the} basis for {the} corresponding continua.

When using RVEs, the main challenge is to define
\[
\int_D \kappa \nabla \phi_{i,\text{fine}}^{\omega_j} \cdot \nabla \phi_{m,\text{fine}}^{\omega_l} \, {dx}
\]
using RVE computations. This is based on {a} localization assumption, which we introduce next.

We consider $\mathcal{H}^{\omega}$, which is {the} harmonic expansion in $\omega$,
which is defined by solving local problems in each $K$. We
 can use
\[
\int_D \kappa \nabla \mathcal{H}^{\omega_j}(\phi_{i,\text{fine}}^{\omega_j}) \cdot \nabla \mathcal{H}^{\omega_l}(\phi_{m,\text{fine}}^{\omega_l}) \, {dx}
\]
to approximate the elements of the stiffness matrix.
Our localization assumption uses the local snapshots computed in
the RVE for each $\omega_i$, which we denote by RVE$_i$. We denote these
RVE snapshots {by} $ \psi^{RVE_i}_{j,\text{fine}}$. Then, we propose
the following localization assumption
\[
\int_D \kappa \nabla \mathcal{H}^{\omega_j}(\phi_{i,\text{fine}}^{\omega_j}) \cdot \nabla \mathcal{H}^{\omega_l}(\phi_{m,\text{fine}}^{\omega_l}) \, {dx}
\approx
\int_D \kappa \nabla \mathcal{H}^{RVE_j}(\chi^{\omega_j}_{\text{fine}} \psi^{RVE_j}_{i,\text{fine}})
\cdot \nabla \mathcal{H}^{RVE_l} (\chi^{\omega_l}_{\text{fine}} \psi^{RVE_l}_{m,\text{fine}}) \, {dx}.
\]

\end{remark}

\subsection{Numerical simulation of the shale gas transport}

In this section, we add a case study for our
method. We follow the example considered in \cite{aevw16}, where
a shale gas transport with dual-continuum (organic and inorganic pores)
(see also \cite{akkutlu2012multiscale}) is studied.
In inorganic matter, we have
\[
\varphi_i \frac{\partial c}{\partial t} =
\div ( (\varphi_i D_i + c ZRT \frac{\kappa_i}{\mu} ) \nabla c) + Q_{ki}.
\]
where $\varphi_i$ is the inorganic porosity, $D_i$ is the tortuosity corrected coefficient of diffusive molecular transport in the inorganic matrix, $\kappa_i$ is the inorganic matrix absolute permeability, $\mu$ is the dynamic gas viscosity, $p_i$ is the inorganic matrix pressure, $p = c ZRT$  and $Q_{ki}$ is the transfer function.

Here, we use the following transfer function
\[
Q_{ki} = \tau_{ki} \sigma (c_k - c), \quad
\tau_{ki} = \varphi_k D_k + (1 - \varphi_k) D_s F'.
\]
where $\sigma$ is a shape factor.

For free and adsorbed gas in the kerogen ($c_k$ and $c_{\mu}$)
\[
(\varphi_k + (1 - \varphi_k) F') \frac{\partial c_k}{\partial t} =
\div ( (\varphi_k D_k + (1 - \varphi_k) D_s F'  + c_k ZRT \frac{\kappa_k}{\mu} ) \nabla c_k) - Q_{ki}.
\]
where $\varphi_k$ is the kerogen porosity, $D_k$ is the tortuosity corrected coefficient of diffusive molecular transport for the free gas in kerogen, $D_s$ is the coefficient of diffusive molecular transport for the adsorbed gas in kerogen, $\kappa_k$ is the kerogen permeability, and $p_k$ is the kerogen pressure. For $c_{\mu}$ we use linear Henry’s isotherm $c_{\mu} = F(c_k)$, $F(c_k) = k_{H} c_k$.

For free-gas in fracture network, we have
\[
\varphi_f \frac{\partial c_f}{\partial t} =
\div (c_f ZRT \frac{\kappa_f}{\mu} \nabla c_f).
\]
where $\varphi_f$ is the fracture porosity, $K_L$ is the diffusion coefficient,
$\kappa_f$ is the fracture absolute permeability, and $p_f$ is the fracture pressure.

\begin{figure}[!ht]
  \centering
  \includegraphics[width=0.6 \textwidth]{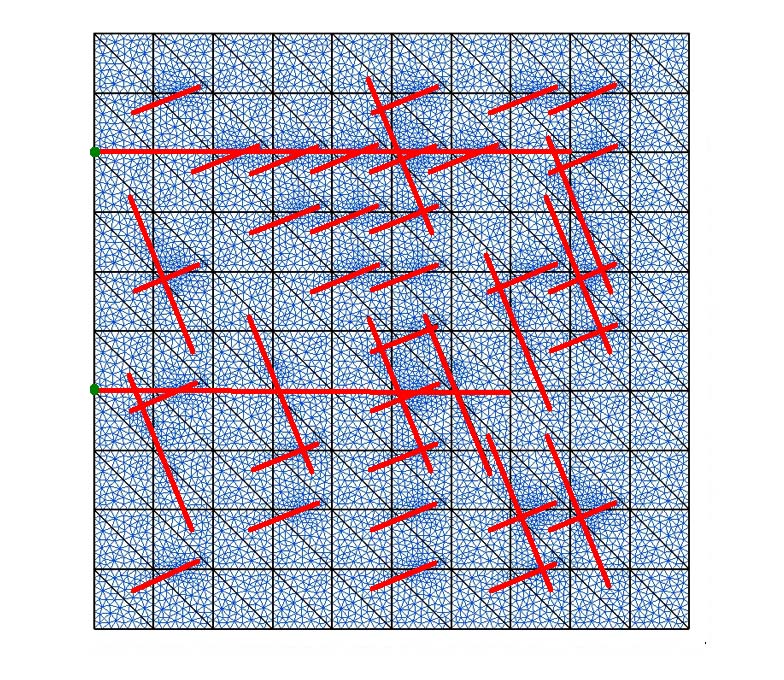}
  \caption{Computational mesh with fractures for shale gas transport. }
  \label{dp-sh-mesh}
\end{figure}

We consider the model geometry with discrete fracture distribution as shown in Figure \ref{dp-sh-mesh}. The coarse grid is uniform and contains $121$ vertices and $200$ coarse cells. The domain $D$ has a length of 50 meters in both directions. The other model parameters used are as follows.
 $R = 8.31$[J/(K $\cdot$ mol)], $T = 323.0$[K], $Z = 1.0$,
 $p_i = 20 \cdot  10^6$[Pa], $p_{well} = 5 \cdot  10^6$[Pa], $p_L = 10^6$[Pa],
 $c_{init} = p_i/(ZRT)$[mol/m$^3$], $c_{well} = p_{well}/(ZRT)$[mol/m$^3$],
 $\varphi_i = 0.025$, $\varphi_k = 0.025$, $\varphi_f = 0.01$
 $\kappa_i = 10^{-19}$[m$^2$], $\kappa_k = 0.0$[m$^2$], $\kappa_{nf} = 10^{-14}$[m$^2$], $\kappa_{hf} = 10^{-13}$[m$^2$],
 $D_s = D_i = D_k = 10^{-8}$[m$^2$/s],
 $k_H = 0.1$, $\mu = 10^{-5}$[Pa $\cdot$ s].
For transfer functions, we set $\sigma = 10.0$ [1/m$^2$].

As we remarked, the purpose of this example is to show the geo-application of our approach. In Figure \ref{sol-dp-sh}, we depict the solutions at the final time $T_{max} = 500$ days.
We observe that the GMsFEM with simplified basis functions provides
a good agreement. In this case, we have observed less than
$1$\% in $L^2$ norm.

\begin{figure}[!ht]
  \centering
  \includegraphics[width=0.49 \textwidth]{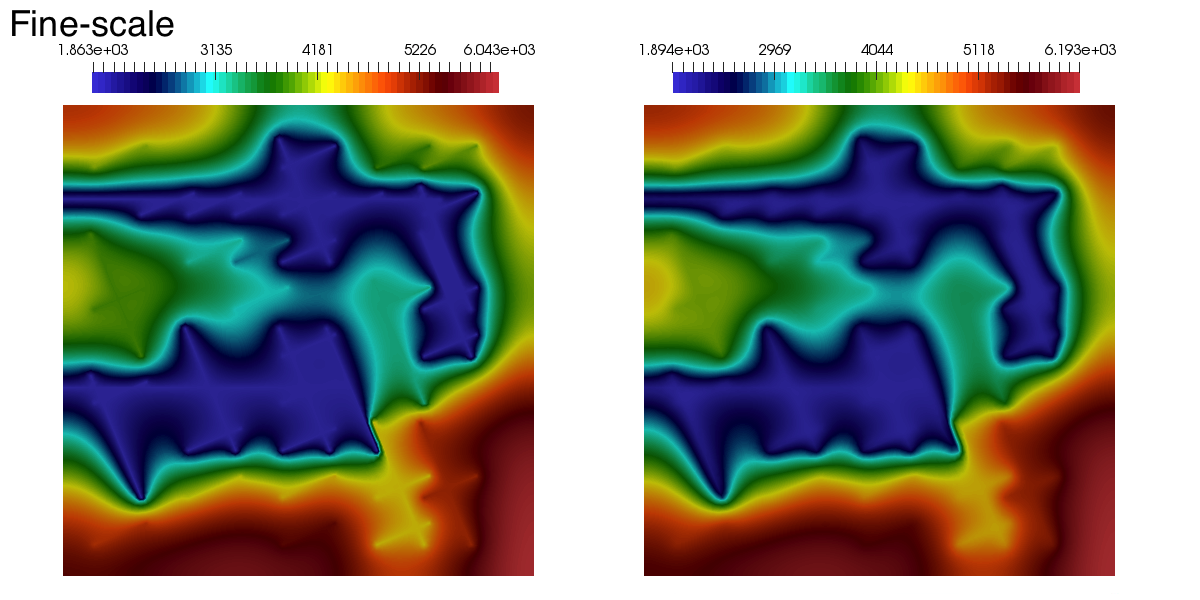}
  \includegraphics[width=0.49 \textwidth]{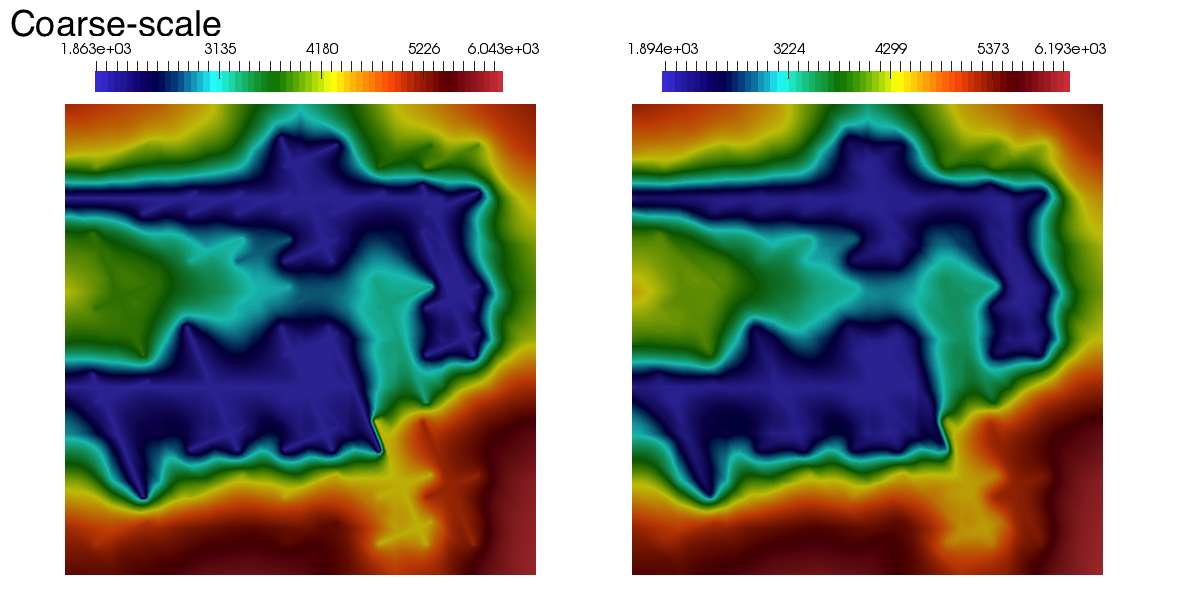}
  \caption{Dual-continuum background for shale gas transport.
  Left: Fine-scale solution $DOF = 18064$.
  Right: Coarse-scale solution $DOF = 938$ using simplified basis functions.}
  \label{sol-dp-sh}
\end{figure}

\section{Conclusions}

In this paper, our goals are:
(1) to investigate the GMsFEM for fractured media;
(2) to study the relation between the GMsFEM and the multi-continuum
approaches; (3) to develop a coupled GMsFEM and multi-continuum
approaches for highly heterogeneous fractured media.
First, we show that GMsFEM basis functions represent each continuum
and these multiscale basis functions correspond to
the eigenvectors associated with very small eigenvalues.
We propose simplified basis functions when fracture geometries are
simple. Multiscale basis functions contain the spatial information
representing the interaction between
the matrix and the fractures. Numerical results show that
the GMsFEM can provide an accurate solution if we include
multiscale basis functions corresponding to very small eigenvalues.
The latter represents the number of the continua in each coarse block.
In the second part of the paper, we develop a coupled GMsFEM and
multi-continuum approaches. In this case, fractures at the fine-subgrid
are represented by {a} multi-continuum approach. As a result,
{the GMsFEM is needed} for a system of equations. In this case, we use un-coupled
and coupled basis functions. In the latter, the multiscale basis functions
are constructed for subgrid multi-continuum media in a coupled fashion.
We present numerical results, where we compute the parameters of the multi-continua
from a subgrid problem. Our numerical results show that {the GMsFEM is able to give solutions
with good accuracy.}
The accuracy is better when using coupled basis functions.

\appendix

\section{Convergence analysis}
\label{appendix}

In this appendix, we present the convergence analysis of our schemes.
We will consider both the un-coupled multiscale basis functions
and the coupled multiscale basis functions
as well as an abstract formulation to be defined in the following.
Note that the abstract formulation can be applied to the practical cases presented in this paper.
We consider the $N$-continuum problem: find $u=(u_{1},u_{2},\cdots u_{N})$ such that $u_i(t,\cdot) \in H^1(\Omega)$, $i=1,\cdots, N$, and
\begin{align}
\label{eq:ref}
\sum_{i}c_i \left(\cfrac{\partial u_i}{\partial t},v_i \right) & =-\sum_{i}a_{i}(u_{i},v_{i})+q(u,v)+(f,v),  \quad\quad t\in(0,T)
\end{align}
for all test functions $v=(v_1,v_2,\cdots, v_N)$ with $v_i(t,\cdot)\in H^1_0(\Omega)${,} where
\begin{equation*}
\begin{split}
c_i(u,v)&=\int_{D_m}c_i \,u \, v \, {dx} +  \sum_j \int_{D_{f,j}}c_{j,i} \, u \, v  \, {dx},  \\
q(u,v)&=\sum_{j}\sum_{i\neq j}Q_{i}\int_{D} (u_{i}-u_{j})v_{j}  \, {dx}, \\
a_{i}(u,v)&=\int_{D_m}\kappa_{i}\nabla u\cdot\nabla v  \, {dx}+\sum_j \int_{D_{f,j}}\kappa_{j,i}\nabla_{f}u\cdot\nabla_{f}v  \, {dx}.
\end{split}
\end{equation*}
{Note} that all the summations are summing over all continu{a}, that is, they are summing over $i,j=1,2,\cdots, N$.
Next we define two global bilinear operators $c(\cdot,\cdot)$ and $a(\cdot,\cdot)$ by
\begin{equation*}
c(u,v)=\sum_i c_i(u_i,v_i), \quad\quad a(u,v) = \sum_i a_i(u_i,v_i).
\end{equation*}
Clearly, we have $q(u,v)=q(v,u)$ and $q(u,u)\leq0$ for all $u(t,\cdot),v(t,\cdot) \in[H^{1}(\Omega)]^{N}$.
{Equation} (\ref{eq:ref}) defines our multi-continuum problem.

We next define the operator $a_{i}^{(j)}(\cdot,\cdot)$ by
\begin{equation*}
a_{i}^{(j)}(u,v)=\left(\int_{\omega_{j}}\kappa_{i}\nabla u\cdot\nabla v  \, {dx}+\sum_l \int_{D_{f,l}\cap \omega_j}\kappa_{l,i}\nabla_{f}u\cdot\nabla_{f}v  \, {dx} \right)
\end{equation*}
for all $u(t,\cdot),v(t,\cdot) \in H_{0}^{1}(\omega_{j})$.
This operator corresponds {to} the contribution of $a_i(u,v)$ {in} the coarse region $\omega_j$. We also define
the corresponding global operator
\begin{equation*}
a^{(j)}(u,v)=\sum_{i}a_{i}^{(j)}(u_{i},v_{i}){.}
\end{equation*}

Finally,
we define two bilinear operators $a_{Q}^{(j)}(\cdot,\cdot)$
and $a_{Q}(\cdot,\cdot)$ by
\begin{equation*}
a_{Q}^{(j)}(u,v)=a^{(j)}(u,v)-q(u,v), \quad\quad a_{Q}(u,v)=a(u,v)-q(u,v)
\end{equation*}
for all $u(t,\cdot),v(t,\cdot) \in H_{0}^{1}(\omega_{j})$.

In the following, we will present the definitions
of the un-coupled multiscale basis functions
and the coupled multiscale basis functions.
For each case, we follow the general procedure
to first construct a local snapshot space for each coarse region $\omega_j$,
and then construct an offline space (consisting of multiscale basis functions) using a suitable spectral problem
defined on the snapshot space.
Note that the snapshot functions and the basis functions are independent of time.

\subsection*{Coupled GMsFEM (snapshot space)}

For each coarse region $\omega_{j}$, we obtain the $k$-th snapshot function by
solving the following local problem:
find $\psi_{k}^{(j),\text{snap}}\in[V_{h}(\omega_{j})]^{N}$ such
that
\begin{align*}
a^{(j)}(\psi_{k}^{(j),\text{snap}},v)-q(\psi_{k}^{(j),\text{snap}},v) & =0,\quad\forall v\in[V_{h,0}(\omega_{j})]^{N}, \\
\text{with the boundary condition}\quad \psi_{k}^{(j),\text{snap}} & =\delta_{k}, \quad \text{on } \partial\omega_j,
\end{align*}
where
$V_{h}(\omega_{j})$ is a fine{-}scale space and $V_{h,0}(\omega_{j})$
is the subspace of $V_{h}(\omega_{j})$ containing functions with zero trace on the boundary of $\omega_j$.
In the above definition, the discrete delta function $\delta_{k}$
is defined as $\delta_{k} = (\delta_{k,1}, \delta_{k,2},\cdots, \delta_{k,N})$
and each $\delta_{k,i}$ is the discrete delta function such that $\delta_{k,i} = 1$  at the fine{-}grid node $x_k\in\partial\omega_j$ and $\delta_{k,i} = 0$
at all other fine{-}grid nodes on $\partial\omega_j$.
Using the above snapshot functions,
we can define the local snapshot space by
\begin{equation*}
V_{\text{snap}}(\omega_{j})=\text{span}\{\psi_{k}^{(j),\text{snap}} \; : \; \forall k\}.
\end{equation*}

\subsection*{Coupled GMsFEM (offline space)}

We will construct the offline space in this section. The offline space is spanned by all multiscale basis functions.
To find the multiscale basis functions, we use a local spectral problem defined in the snapshot space.
More precisely,
for each coarse region $\omega_{j}$, we consider the following local eigenvalue
problem: find the $k$-th eigenfunction
$\phi_{k}^{(j)}\in V_{\text{snap}}(\omega_{j})$ and the $k$-th eigenvalue $\lambda^{(j)}_k$
such that
\begin{align*}
a_Q^{(j)}(\phi_{k}^{(j)},v) & =\lambda_{i}^{(j)} \, s^{(j)}(\phi_{k}^{(j)},v), \quad\;\forall v\in V_{\text{snap}}(\omega_{j}),
\end{align*}
where the bilinear form $s^{(j)}$ is defined as
\begin{equation*}
s^{(j)}(u,v)=\sum_{i}\left(\int_{\omega_{j}}\kappa_{i}|\nabla\chi_{j}|^{2}u \, v \, {dx}+\sum_l \int_{D_{f,l}}\kappa_{l,i}|\nabla_{f}\chi_{j}|^{2} u \, v \, {dx} \right)
\end{equation*}
and the eigenvalues are arranged in ascending order.
Using the eigenfunction $\phi_{k}^{(j)}$, we can define the $k$-th multiscale basis function
by $\hat{\phi}_{k}^{(j)}=\chi_{j} \phi_{k}^{(j)}$,
where $\{ \chi_j\}$ is a set of partition of unity functions
for the coarse{-}grid partition of the domain $\Omega$.
Finally,
the
local offline space is defined by $V_{H}(\omega_{j})=\text{span}\{\hat{\phi}_{i}^{(j)}|\; i\leq L_j\}$,
which is formed by using the first $L_j$ eigenfunctions.
In addition,
the global offline space, $V_{H}$, is defined by $V_{H}=\sum_j V_{H}(\omega_{j})$.

\subsection*{Un-Coupled GMsFEM (snapshot space)}

Now, we will present the construction of {the} basis for the un-coupled case.
We first consider the construction of the snapshot space.
For each coarse region $\omega_{j}$ and for each continuum $i$,
we obtain the $k$-th snapshot function by
solving the problem:
find $\psi_{k,i}^{(j),\text{snap}}\in V_{h}(\omega_{j})$ such that
\begin{align*}
a_{i}^{(j)}(\psi_{k,i}^{(j),\text{snap}},v) & =0, \quad \;\forall v\in V_{h,0}(\omega_{j}),\\
\text{with the boundary condition}\quad\psi_{k,i}^{(j),\text{snap}} & =\delta_{k,i}, \quad \text{on } \partial\omega_j.
\end{align*}
Then,
 the local snapshot space for the coarse region $\omega_j$ and for the
 $i$-th continuum
 is defined by
 \begin{equation*}
  V_{\text{snap}}^{(i)}(\omega_{j})=\text{span}\{\psi_{k,i}^{(j),\text{snap}} \; : \; \forall k \}, \quad i=1,2,\cdots, N.
\end{equation*}

\subsection*{Un-Coupled GMsFEM (offline space)}

We will construct multiscale basis functions for each coarse region $\omega_j$
and for each continuum $i$. To do so, we consider
the following local eigenvalue
problem: find the $k$-th eigenfunction $\phi_{k,i}^{(j)}\in V_{\text{snap}}^{(i)}(\omega_{j})$
and the $k$-th eigenvalue $\lambda_{k,i}^{(j)}$
such that
\begin{align*}
a_{i}^{(j)}(\phi_{k,i}^{(j)},v) & =\lambda_{k,i}^{(j)}\, s_{i}^{(j)}(\phi_{k,i}^{(j)},v), \quad\;\forall v\in V_{\text{snap}}^{(i)}(\omega_{j}),
\end{align*}
where
\begin{equation*}
s_{i}^{(j)}(u,v)=\int_{\omega_{j}}\kappa_{i}|\nabla\chi_{j}|^{2}u \, v \, {dx} +\sum_l \int_{D_{f,l}}\kappa_{l,i}|\nabla_{f}\chi_{j}|^{2}u \, v \, {dx}.
\end{equation*}
We assume that the eigenvalues are arranged in ascending order.
Using the above eigenfunctions, we can define the $k$-th multiscale basis function by $\hat{\phi}_{k,i}^{(j)}=\chi_{j}\phi_{k,i}^{(j)}$.
To define the offline space for the $i$-th continuum and for the coarse region $\omega_j$, we take the first $L_j$ eigenfunctions
and define $V_{H}^{(i)}(\omega_{j})=\text{span}\{\hat{\phi}_{k,i}^{(j)}|\; k\leq L_j\}$.
Note that $L_j$ can depend on $i$, but we omit this index to simplify the notation.
Then the global offline space for the $i$-th continuum is given by $V_{H}^{(i)}=\sum_j V_{H}^{(i)}(\omega_{j})$.
Finally, the offline space, $V_{H}$, is defined by $V_{H}=V_{H}^{1}\times V_{H}^{2}\times\cdots\times V_{H}^{N}$.

\subsection*{Analysis}

Now we are ready to present the analysis. We will first prove
the following best approximation estimates (see Lemma \ref{lem:best} and Lemma \ref{lem:best2}).
We will compare the difference between the reference solution $u$, defined by (\ref{eq:ref}),
and the multiscale solution $u_{ms}\in V_H$ defined by
\begin{align}
\label{eq:ms}
\sum_{i}c_i \left(\cfrac{\partial u_{ms,i}}{\partial t},v_i \right) & =-\sum_{i}a_{i}(u_{ms,i},v_{i})+q(u_{ms},v)+(f,v), \;\forall v\in V_H, \, t\in(0,T).
\end{align}
We also define the following norms
\begin{equation*}
\|u\|_c^2 =  c(u,u), \quad \|u\|_a^2 = a(u,u), \quad
\|u\|_{a_Q}^2 = a_Q(u,u).
\end{equation*}

\begin{lemma}
\label{lem:best}
Let $u$ be the reference solution defined in (\ref{eq:ref}) and $u_{ms}$ be the multiscale
numerical solution defined in (\ref{eq:ms}). We have
\begin{equation}
\label{lem:best1}
\begin{split}
&\: \|u(t,\cdot)-u_{ms}(t,\cdot)\|_{c}^2+\int_{0}^{T}\|u-u_{ms}\|_{a_{Q}}^{2} {dt}  \\
\leq &\: C\inf_{w\in V_{H}}\Big(\int_{0}^{T}\Big\|\cfrac{\partial(w-u)}{\partial t}\Big\|_{c}^2 {dt} +
\int_{0}^{T}\|w-u\|_{a_{Q}}^{2} {dt}  +\|w(0,\cdot)-u(0,\cdot)\|_{c}^2 \Big).
\end{split}
\end{equation}
\end{lemma}
\begin{proof}
We write $u_{ms} = (u_{ms,1},\cdots, u_{ms,N})${,} where $u_{ms,i}$ is the component
for the $i$-th continuum.
Using (\ref{eq:ref}) and (\ref{eq:ms}), we have
\begin{align*}
c \left(\cfrac{\partial(u-u_{ms})}{\partial t}, v \right)+\sum_{i}a_{i}(u_{i}-u_{ms,i},v)-q(u-u_{ms},v) & =0,\quad\;\forall v\in V_{H}, \, t\in(0,T).
\end{align*}
Let $w\in V_H$ and $v=w-u_{ms}$ in the above equation, we obtain
\begin{align*}
 & \: c \left(\cfrac{\partial(w-u_{ms})}{\partial t},w-u_{ms} \right)+\sum_{i}a_{i}(w_{i}-u_{ms,i},w_{i}-u_{ms,i})-q(w-u_{ms},w-u_{ms})\\
= &\: c \left(\cfrac{\partial(w-u)}{\partial t},w-u_{ms} \right)+\sum_{i}a_{i}(w-u_{i},w-u_{ms,i})-q(w-u,w-u_{ms})\\
\leq & \: \Big \|\cfrac{\partial(w-u)}{\partial t}\Big\|_{c}\|w-u_{ms}\|_{c}+\|w-u\|_{a_{Q}}\|w-u_{ms,}\|_{a_{Q}}.
\end{align*}
Therefore, integrating the above in time,
we obtain (\ref{lem:best1}).
\end{proof}

In the next lemma, we prove a similar result as (\ref{lem:best1}) by
assuming an additional condition on $q$, namely,
\begin{equation}
\label{eq:assume}
-q(v,v)\leq D\|v\|_{a}^{2}, \quad \forall \quad
v\in[H^{1}(\Omega)]^{N}.
\end{equation}

\begin{lemma}
\label{lem:best2}
Assume that $-q(v,v)\leq D\|v\|_{a}^{2}, \forall
 v\in[H^{1}(\Omega)]^{N}$.
For the same $u$ and $u_{ms}$ as in Lemma \ref{lem:best},
we have
\begin{align*}
&\: \|u(t,\cdot)-u_{ms}(t,\cdot)\|_{c}^{2}+\int_{0}^{T}\|u-u_{ms}\|_{a}^{2} \, {dt} \\
\leq &\: C\inf_{w\in V_{H}} \Big(\int_{0}^{T}\Big\|\cfrac{\partial(w-u)}{\partial t}\Big\|_{c}^{2} \, {dt}+
(D+1)\int_{0}^{T}\|w-u\|_{a}^{2} \, {dt} +\|w(0,\cdot)-u(0,\cdot)\|_{c}^{2}\Big).
\end{align*}
\end{lemma}

\begin{proof}
Since $\int_{0}^{T}\|u-u_{ms}\|_{a}^{2} \, {dt} \leq \int_{0}^{T}\|u-u_{ms}\|_{a_{Q}}^{2} \, {dt}$
and $-q(v,v)\leq D\|v\|_{a}^{2}$, we have
\begin{align*}
&\: \|u(t,\cdot)-u_{ms}(t,\cdot)\|_{c}^{2}
+\int_{0}^{T}\|u-u_{ms}\|_{a}^{2} \, {dt} \\
\leq
&\: \|u(t,\cdot)-u_{ms}(t,\cdot)\|_{c}^{2}
+\int_{0}^{T}\|u-u_{ms}\|_{a_{Q}}^{2} \, {dt}\\
  \leq
&\: C\inf_{w\in V_{H}}\Big(\int_{0}^{T}\Big\|\cfrac{\partial(w-u)}{\partial t}\Big\|_{c}^{2} \, {dt}
+\int_{0}^{T}\|w-u\|_{a}^{2}-q(w-u,w-u) \, {dt}
+\|w(0,\cdot)-u(0,\cdot)\|_{c}^{2}\Big)\\
  \leq
&\: C\inf_{w\in V_{H}}\Big(\int_{0}^{T}\Big\|\cfrac{\partial(w-u)}{\partial t}\Big\|_{c}^{2} \, {dt}
+(D+1)\int_{0}^{T}\|w-u\|_{a}^{2} \, {dt}
+\|w(0,\cdot)-u(0,\cdot)\|_{c}^{2}\Big).
\end{align*}
  This completes the proof.
\end{proof}


We will use the above two lemmas to prove the convergence of our scheme.
In particular, we need {to} find a suitable function $w\in V_H$
and estimate the difference $w-u$ in various norms.
The following is our strategy. We define the snapshot projection $u_{snap} \in V_{snap}$ by
\begin{align}
\label{eq:ms1}
u_{snap} = \sum_j \chi_j u^{(j)}_{snap},\quad \text{ {with} }  \; u^{(j)}_{snap}|_{\partial \omega_j} = u|_{\partial \omega_j} ,
\end{align}
where $V_{snap}$ is the snapshot space obtained by collecting all snapshot functions.
We note that, since the snapshot functions for each coarse region $\omega_j$
take all possible values on $\partial\omega_j$, the {problem} in (\ref{eq:ms1}) is well-defined.
Since
$w-u = w-u_{snap} + u_{snap}-u$, it suffices to estimates
the two terms $w-u_{snap}$ and $u_{snap}-u$.
{Note} that the term $u_{snap}-u$ corresponds {to} an irreducible error of our scheme, since this error
cannot be improved by using our scheme. We assume that this irreducible error is small
by using a large set of snapshot functions. Based on this argument,
it suffices to estimate $w-u_{snap}$ by choosing an appropriate function $w\in V_H$.

{Note} that $u_{snap}$ is in the snapshot space, which means that we can represent $u$
as a linear combination of all multiscale basis functions. To define $w\in V_H$,
we will take $w$ as the projection of $u_{snap}$ in the offline space. More precisely,
we use the following construction.
First, for the case of un-coupled basis functions, we can represent
\begin{equation}
\label{eq:snap1}
u_{snap} = (u_{snap,1},u_{snap,2},\cdots,u_{snap,N}), \quad
u_{snap,i} = \sum_{j}\sum_{k}c_{k,i}^{(j)}(t)\chi_{j}(x)\phi_{k,i}^{(j)}(x).
\end{equation}
Then the projection $w$ of $u$ in the offline space is defined as
\begin{equation}
\label{eq:off1}
w = (w_1,w_2,\cdots,w_N), \quad
w_i = \sum_{j}\sum_{k \leq L_j}c_{k,i}^{(j)}(t)\chi_{j}(x)\phi_{k,i}^{(j)}(x).
\end{equation}
Second, for the case of coupled basis functions, we can represent
\begin{equation}
\label{eq:snap2}
u_{snap}= \sum_{j}\sum_{k}c_{k}^{(j)}(t)\chi_{j}(x)\phi_{k}^{(j)}(x).
\end{equation}
Then the projection $w$ of $u_{snap}$ in the offline space is defined as
\begin{equation}
\label{eq:off2}
w = \sum_{j}\sum_{k \leq L_j}c_{k}^{(j)}(t)\chi_{j}(x)\phi_{k}^{(j)}(x).
\end{equation}

Next, we will state and prove the main results (Theorem \ref{thm1} and Theorem \ref{thm2})
of this appendix.
As we will see, Theorem \ref{thm1} and Theorem \ref{thm2}
follow from Lemmas \ref{lemma1}, \ref{lemma2}{,} and \ref{lemma3}.

\begin{theorem}
\label{thm1}
For the un-coupled GMsFEM, let $u$ and $u_{snap}$ be the reference solution and snapshot projection in (\ref{eq:ref}) and (\ref{eq:ms})
and let $w\in V_{H}$ be the projection of $u_{snap}$ defined in (\ref{eq:off1}). We assume (\ref{eq:assume}).
Then we have
\begin{equation*}
\begin{split}
&\: \int_{0}^{T}\Big\|\cfrac{\partial(w-u_{snap})}{\partial t}\Big\|_{c}^{2} \, {dt}
+\int_{0}^{T}\|w-u_{snap}\|_{a}^{2} \, {dt}
+\|w(0,\cdot)-u_{snap}(0,\cdot)\|_{c}^{2} \\
\leq
&\: \cfrac{C}{\Lambda_1}\left(\int_{0}^{T}\Big\|\cfrac{\partial u}{\partial t}\Big\|_{a}^{2} \, {dt}
+\int_{0}^{T}\|u\|_{a}^{2} \, {dt}
+\|u(0,\cdot)\|_{a}^{2}\right),
\end{split}
\end{equation*}
where $\Lambda_1=\min_{j,i}\{\lambda_{L_j+1,i}^{(j)}\}$.
\end{theorem}

\begin{theorem}
\label{thm2}
For the coupled GMsFEM, let $u$ and $u_{snap}$ be the reference solution and snapshot projection in (\ref{eq:ref}) and (\ref{eq:ms})
and let $w\in V_{H}$ be the projection of $u_{snap}$ defined in (\ref{eq:off2}).
Then
we have
\begin{equation*}
\begin{split}
&\: \int_{0}^{T}\Big\|\cfrac{\partial(w-u_{snap})}{\partial t}\Big\|_{c}^{2} \, {dt}
+\int_{0}^{T}\|w-u_{snap}\|_{a_{Q}}^{2} \, {dt}
+\|w(0,\cdot)-u_{snap}(0,\cdot)\|_{c}^{2} \\
\leq
&\: \cfrac{C^2}{\Lambda_2}\left(\int_{0}^{T}\Big\|\cfrac{\partial u}{\partial t}\Big\|_{a_{Q}}^{2} \, {dt}
+\int_{0}^{T}\|u\|_{a_{Q}}^{2} \, {dt}
+\|u(0,\cdot)\|_{a_{Q}}^{2}\right),
\end{split}
\end{equation*}
where $\Lambda_2=\min_{j}\{\lambda_{L_j+1}^{(j)}\}$.
\end{theorem}
We will proof the above two theorems by estimating
$\int_{0}^{T} \Big\|\cfrac{\partial(w-u_{snap})}{\partial t} \Big\|_{c}^{2} \, {dt}$,
$\int_{0}^{T}\|w-u_{snap}\|_{a}^{2}$, $\int_{0}^{T}\|w-u_{snap}\|_{a_{Q}}^{2} \, {dt}${,} and $\|w(0,\cdot)-u_{snap}(0,\cdot)\|_{c}^{2}$ separately in the following lemmas.
Unless otherwise specified, the constant $C$ is independent of any scales and continuum.

\begin{lemma}
\label{lemma1}
Let $u$, $u_{snap}${,} and $w$ be defined as in Theorems \ref{thm1} and \ref{thm2}. For the un-coupled basis functions,
we have
\[
\int_{0}^{T}\Big\|\cfrac{\partial(w-u_{snap})}{\partial t}\Big\|_{c}^{2}  \, {dt}
\leq\cfrac{CE}{\Lambda_1}\, \Big\|\cfrac{\partial u}{\partial t}\Big\|_{a}^{2}  \, {dt}.
\]
For the coupled basis functions, we have
\[
\int_{0}^{T}\Big\|\cfrac{\partial(w-u_{snap})}{\partial t}\Big\|_{c}^{2}
\leq\cfrac{CE}{\Lambda_2}\, \Big\|\cfrac{\partial u}{\partial t}\Big\|_{a_{Q}}^{2},
\]
where $E=\max_{i,j,l}\Big\{\cfrac{c_{i}\chi_{j}^{2}}{\kappa_{i}|\nabla\chi_{j}|^{2}},\, \cfrac{c_{l,i}\chi^{2}_{j}}{\kappa_{l,i}|\nabla_{f}\chi_{j}|^2}\Big\}$.
\end{lemma}
\begin{proof}
We will present the proof for the case of un-coupled basis functions. First, {note} that
\begin{eqnarray*}
\|(u_{snap})_{t}-w_{t}\|_{c}^{2} &
\leq
& \sum_{i}\Big\|\sum_{j}\left(\chi_{j}\cfrac{\partial u^{(j)}_{snap,i}}{\partial t}
-\sum_{k\leq L_j}\cfrac{\partial c_{k,i}^{(j)}}{\partial t}\chi_{j} \phi_{k,i}^{(j)} \right)\Big\|_{c}^{2}\\
 &
 \leq
 & D\sum_{i}\sum_{j}\int_{\omega_{j}}\cfrac{c_i\chi_{j}^{2}}{\kappa_{i}|\nabla\chi_{j}|^{2}}\kappa_{i}|\nabla\chi_{j}|^{2}\left(\cfrac{\partial u^{(j)}_{snap,i}}{\partial t}-\sum_{k\leq L_j}\cfrac{\partial c_{k,i}^{(j)}}{\partial t} \, \phi_{k,i}^{(j)}\right)^{2}  \, {dx}\\
 & &
 +\sum_{i,j,l}\int_{D_{f,l}\cap\omega_j} c_{l,i}\cfrac{\chi^{2}_{j}}{\kappa_{l,i}|\nabla_{f}\chi_{j}|^2}\kappa_{l,i}|\nabla_{f}\chi_{j}|^2 \left(\cfrac{\partial u^{(j)}_{snap,i}}{\partial t}-\sum_{k\leq L_j}\cfrac{\partial c_{k,i}^{(j)}}{\partial t} \, \phi_{k,i}^{(j)}\right)^{2}  \, {dx} \\
 &
 \leq & DE\sum_i
 \sum_{j}s_i^{(j)} \left(\sum_{k>L_j}\cfrac{\partial c_{k,i}^{(j)}(t)}{\partial t}\phi_{k,i}^{(j)},\sum_{k>L_j}\cfrac{\partial c_{k,i}^{(j)}(t)}{\partial t}\phi_{k,i}^{(j)} \right)
\end{eqnarray*}
with $D=\max_{K\in\mathcal{T}^H} \{D_K\}$ where $D_K$ is the number of coarse neighborhoods intersecting with $K$.
By using the orthogonality of eigenfunctions, we have
\begin{align*}
s_i^{(j)} \left(\sum_{k>L_j}\cfrac{\partial c_{k,i}^{(j)}(t)}{\partial t}\phi_{k,i}^{(j)},\sum_{k>L_j}\cfrac{\partial c_{k,i}^{(j)}(t)}{\partial t}\phi_{k,i}^{(j)} \right)
&
\leq
\sum_{k>L_j} \frac{1}{\lambda_{k,i}^{(j)}} \left(\cfrac{\partial c_{k,i}^{(j)}(t)}{\partial t}\right)^{2} a_i^{(j)}(\phi_{k,i}^{(j)},\phi_{k,i}^{(j)})\\
 &
 \leq
 \cfrac{1}{\lambda_{L_j+1,i}^{(j)}}\sum_{k}\left(\cfrac{\partial c_{k,i}^{(j)}(t)}{\partial t}\right)^{2} a_i^{(j)}(\phi_{k,i}^{(j)},\phi_{k,i}^{(j)})\\
 &
 =\cfrac{1}{\lambda_{L_j+1,i}^{(j)}} a_i^{(j)} \left(\cfrac{\partial u^{(j)}_{snap,i}}{\partial t},\cfrac{\partial u^{(j)}_{snap,i}}{\partial t} \right).
\end{align*}
Since $u^{(j)}_{snap,i}$ is the $a^{(j)}_i$-harmonic expansion of $u_i$ in $\omega_j$, we have
\[
a^{(j)}_i(u^{(j)}_{snap},u^{(j)}_{snap}) \leq a^{(j)}_i(u_i,u_i)
\]
and similarly
\[
a^{(j)}_i \left(\cfrac{\partial u^{(j)}_{snap,i}}{\partial t},\cfrac{\partial u^{(j)}_{snap,i}}{\partial t} \right)
\leq
a^{(j)}_i \left(\cfrac{\partial u_{i}}{\partial t},\cfrac{\partial u_{i}}{\partial t} \right).
\]
Therefore, by summing over all $i,j$, we obtain
\begin{align*}
\sum_{i,j}s_i^{(j)} \left(\sum_{k>L_j}\cfrac{\partial c_{k,i}^{(j)}(t)}{\partial t}\phi_{k,i}^{(j)},\sum_{k>L_j}\cfrac{\partial c_{k,i}^{(j)}(t)}{\partial t}\phi_{k,i}^{(j)} \right)
&
\leq
\sum_{i,j}\cfrac{1}{\lambda_{L_j+1,i}^{(j)}} a_i^{(j)} \left(\cfrac{\partial u_i}{\partial t},\cfrac{\partial u_i}{\partial t} \right)\\
&
\leq
\cfrac{1}{\min_{i,j}\{\lambda_{L_j+1,i}^{(j)}\}}\sum_{i,j} a_i^{(j)} \left(\cfrac{\partial u_i}{\partial t},\cfrac{\partial u_i}{\partial t} \right)\\
&
\leq
\cfrac{D}{\min_{i,j}\{\lambda_{L_j+1,i}^{(j)}\}}a \left(\cfrac{\partial u}{\partial t},\cfrac{\partial u}{\partial t} \right).
\end{align*}

For the case of coupled basis functions, we have $s^{(j)}(\cdot,\cdot)=\sum_i s^{(j)}_i(\cdot,\cdot)$.
By using the same arguments, we have

\[
\|(u_{snap})_{t}-w_{t}\|_{c}^{2} \leq
DE \sum_{j}s^{(j)} \left(\sum_{k>L_j}\cfrac{\partial c_{k}^{(j)}(t)}{\partial t}\phi_{k}^{(j)},\sum_{k>L_j}\cfrac{\partial c_{k}^{(j)}(t)}{\partial t}\phi_{k}^{(j)} \right)
\]
and
\[
s^{(j)} \left(\sum_{k>L_j}\cfrac{\partial c_{k}^{(j)}(t)}{\partial t}\phi_{k}^{(j)},\sum_{k>L_j}\cfrac{\partial c_{k}^{(j)}(t)}{\partial t}\phi_{k}^{(j)} \right)
 \leq
 \cfrac{1}{\lambda_{L_j+1}^{(j)}} a^{(j)}_{Q} \left(\cfrac{\partial u^{(j)}_{snap}}{\partial t},\cfrac{\partial u^{(j)}_{snap}}{\partial t} \right).
\]
 Since $u^{(j)}_{snap}$ is the $a^{(j)}_{Q}$-harmonic expansion of $u_i$ in $\omega_j$, we have
\[
a^{(j)}_{Q}(u^{(j)}_{snap},u^{(j)}_{snap}) \leq a^{(j)}_{Q}(u,u)
\]
and
\[
a^{(j)}_{Q} \left(\cfrac{\partial u^{(j)}_{snap}}{\partial t},\cfrac{\partial u^{(j)}_{snap}}{\partial t} \right)
\leq
a^{(j)}_{Q} \left(\cfrac{\partial u}{\partial t},\cfrac{\partial u}{\partial t} \right).
\]
Therefore the proof is complete.
\end{proof}

Before we estimate the terms $\int_{0}^{T}\|w-u_{snap}\|_{a}^{2} \, {dt}$ and $\int_{0}^{T}\|w-u_{snap}\|_{a_{Q}}^{2} \, {dt}$,
we first prove the following lemma.

\begin{lemma}
\label{lemma0}
For the case of coupled basis functions, if $u$ satisfies
\begin{equation*}
\sum_{i}\int_{\omega_{j}}\kappa_{i}\nabla u_{i}\cdot\nabla v_{i} \, {dx}
+\sum_l \int_{D_{f,l}\cap \omega_j}\kappa_{l,i}\nabla_{f}u_{i}\cdot\nabla_{f}v_{i} \, {dx}
-q(u,v)
=\int_{\omega_{j}}f v \, {dx}, \quad\forall v\in[H_{0}^{1}(\omega_{j})]^{N},
\end{equation*}
then we have
\begin{align*}
 &\: \sum_{i}\int_{\omega_{j}}\kappa_{i}\chi_{j}^{2}|\nabla u_{i}|^{2} \, {dx}
 +\sum_l \int_{D_{f,l}\cap \omega_j}\kappa_{l,i}\chi_{j}^{2}|\nabla_{f_{j}}u|^{2}-q(\chi_{j}u,\chi_{j}v)  \, {dx} \\
\leq
&\: C\sum_{i}\Big(\int_{\omega}\cfrac{\chi_{j}^{4}}{\kappa_{i}|\nabla\chi_{j}|^{2}}f_{i}^{2}  \, {dx}
+\int_{\omega_{j}}\kappa_{i}|\nabla\chi_{j}|^{2}u^{2} \, {dx}
]+\sum_l \int_{D_{f,l}\cap \omega_j}\kappa_{l,i}|\nabla_{f}\chi_{j}|^{2}u^{2}  \, {dx} \Big).
\end{align*}
For the case of
un-coupled basis functions, if $u$ satisfies
\begin{equation*}
\int_{\omega_{j}}\kappa_{i}\nabla u\cdot\nabla v  \, {dx}
+\sum_l \int_{D_{f,l}\cap \omega_j}\kappa_{l,i}\nabla_{f}u\cdot\nabla_{f}v  \, {dx}
=\int_{\omega_{j}}f v \, {dx}, \quad\forall v\in H_{0}^{1}(\omega_{j}) ,
\end{equation*}
then we have
\begin{align*}
 & \: \sum_{i}\int_{\omega_{j}}\kappa_{i}\chi_{j}^{2}|\nabla u_{i}|^{2}  \, {dx}
 +\sum_l \int_{D_{f,l}\cap \omega_j}\kappa_{l,i}\chi_{j}^{2}|\nabla_{f}u|^{2}  \, {dx} \\
\leq
&\:  C\sum_{i}\Big(\int\cfrac{\chi_{j}^{4}}{\kappa_{i}|\nabla\chi_{j}|^{2}}f_{i}^{2}  \, {dx}
+\int_{\omega_{j}}\kappa_{i}|\nabla\chi_{j}|^{2}u^{2}  \, {dx}
+\sum_l \int_{D_{f,l}\cap \omega_j}\kappa_{l,i}|\nabla_{f}\chi_{j}|^{2}u^{2}  \, {dx}\Big).
\end{align*}
\end{lemma}

\begin{proof}
For the case of coupled basis {functions}, we take $v=\chi_{j}^{2}u$ and obtain
\begin{align*}
\sum_{i}\int_{\omega_{j}}\kappa_{i}\nabla u_{i}\cdot\nabla(\chi_{j}^{2}u_{i}) \, {dx}
+\sum_l \int_{D_{f,l}\cap \omega_j}\kappa_{l,i}\nabla_{f}u_{i}\cdot\nabla_{f}(\chi_{j}^{2}u_{i}) \, {dx}
-q(\chi_{j}u,\chi_{j}u) &
=\int_{\omega_{j}}\chi_{j}^{2}fu  \, {dx}.
\end{align*}
This implies
\begin{align*}
&\: \sum_{i}(\int_{\omega_{j}}\kappa_{i}\chi_{j}^{2}|\nabla u_{i}|^{2} \, {dx}
+\sum_l \int_{D_{l,i}\cap \omega_j}\kappa_{l,i}\chi_{j}^{2}|\nabla_{f}u_{i}|^{2}  \, {dx}
-q(\chi_{j}u,\chi_{j}u))\\
=
& \:\int_{\omega_{j}}\chi_{j}^{2}fu \, {dx}
-2\sum_{i}\left(\int_{\omega_{j}}\kappa_{i}\chi_{j}u_{i}\nabla u_{i}\cdot\nabla\chi_{j} \, {dx}
+\sum_l\int_{D_{f,l}\cap \omega_j}\kappa_{l,i}\chi_{j}u_{i}\nabla_{f} u_{i}\cdot\nabla_{f}\chi_{j}  \, {dx} \right)\\
\leq
&\: C\sum_{i}\left(\int\cfrac{\chi_{j}^{4}}{\kappa_{i}|\nabla\chi_{j}|^{2}}f_{i}^{2} \, {dx}
+\int_{\omega_{j}}\kappa_{i}|\nabla\chi_{j}|^{2}u_{k}^{2} \, {dx}
+\sum_l \int_{D_{f,l}\cap \omega_j}\kappa_{l,i}|\nabla\chi_{j}|^{2}u_{i}^{2} \, {dx}\right).
\end{align*}
This completes the proof for the case of coupled basis functions.
For the case of un-coupled basis functions, the proof is similar and is therefore omitted.
\end{proof}

\begin{lemma}
\label{lemma2}
Let $u$, $u_{snap}$ and $w$ be defined as in Theorem \ref{thm1} and \ref{thm2}. For the case of un-coupled basis functions,
we have
\[
\int_{0}^{T}\|w-u_{snap}\|_{a}^{2} \, {dt}
\leq
\cfrac{C}{\Lambda_{1}}\|u\|_{a}^{2}.
\]
For the case of coupled basis functions, we have
\[
\int_{0}^{T}\|w-u_{snap}\|_{a_{Q}}^{2} \, {dt}
\leq
\cfrac{C}{\Lambda_{2}}\|u\|_{a_{Q}}^{2},
\]
where $\Lambda_1$ and $\Lambda_2$ are defined in Theorem \ref{thm1} and \ref{thm2}.
\end{lemma}

\begin{proof}
We first define $e^{(j)}_{i}$ by
\begin{equation*}
e^{(j)}_{i} =
\begin{cases}
\sum_{k>L_j}\cfrac{\partial c_{k,i}^{(j)}(t)}{\partial t}\phi_{k,i}^{(j)}, &\quad \text{for un-coupled basis {functions}}, \\
{} \\
\sum_{k>L_j}\cfrac{\partial c_{k}^{(j)}(t)}{\partial t}\phi_{k}^{(j)}, &\quad\text{for coupled basis {functions}}.
\end{cases}
\end{equation*}

{Note} that
\begin{align*}
&\quad \|w-u_{snap}\|_{a}^{2} \\
& =\sum_{i}\left(\int_{\omega_{j}}\kappa_{i}\nabla\sum_{j}\chi_{j}e_{i}^{(j)}\cdot\nabla\sum_{j}\chi_{j}e_{i}^{(j)} \, {dx}
+\sum_l \int_{D_{f,l}\cap \omega_j} \kappa_{l,i}\nabla_{f}\sum_{j}\chi_{j}e_{i}^{(j)}\cdot\nabla_{f}\sum_{j}\chi_{j}e_{i}^{(j)} \, {dx} \right)\\
 &
 \leq
 D\sum_{i}\sum_{j}\left(\int_{\omega_{j}}\kappa_{i}|\nabla\chi_{j}e_{i}^{(j)}|^{2} \, {dx}
 +\sum_l \int_{D_{f,l}\cap \omega_j} \kappa_{l,i} |\nabla_{f}\chi_{j}e_{i}^{(j)}|^{2} \, {dx} \right)\\
 & \leq CD\sum_{i,j}\left(\int_{\omega_{j}}\kappa_{i}\chi_{j}^{2}|\nabla e_{i}^{(j)}|^{2} \, {dx}
 +\sum_l \int_{D_{f,l}\cap \omega_j}\kappa_{l,i}\chi_{j}^{2}|\nabla_{f}e_{i}^{(j)}|^{2} \, {dx}
 +s^{(j)}_i(e^{(j)}_i,e^{(j)}_i)\right).
\end{align*}
In addition, we have
\[
-q \left(\sum_{j}\chi_{j}e^{(j)},\sum_{j}\chi_{j}e^{(j)} \right)
\leq
-D\sum_j \left(q(\chi_{j}e^{(j)},\chi_{j}e^{(j)})\right),
\]
where $D$ is defined in the proof of Lemma \ref{lemma1}.

For the case of coupled basis functions, using Lemma \ref{lemma0}, we obtain
\[
\sum_{i}\int_{\omega_{j}}\kappa_{i}|\nabla\chi_{j}e_{i}^{(j)}|^{2} \, {dx}
+ \sum_l \int_{D_{f,l} \cap \omega_j} \kappa_{l,i}|\nabla_{f}\chi_{j}e_{i}^{(j)}|^{2} \, {dx}
-q(\chi_{j}e^{(j)},\chi_{j}e^{(j)})\leq C\, s^{(j)}(e^{(j)},e^{(j)}).
\]
Therefore, we have
\begin{align*}
\|w-u_{snap}\|^2_{a_Q} & \leq \|w-u_{snap}\|^2_a - D\sum_j \left(q(\chi_{j}e^{(j)},\chi_{j}e^{(j)})\right)\\
&\leq C\,s^{(j)}(e^{(j)},e^{(j)}).
\end{align*}
For the case of un-coupled basis functions, using Lemma \ref{lemma0} again, we obtain
\[
\sum_{k}\int_{\omega_{j}}\kappa_{i}|\nabla\chi_{j}e_{k}^{(j)}|^{2} \, {dx}
+ \sum_l \int_{D_{f,l}\cap \omega_j} \kappa_{l,k}|\nabla_{f}\chi_{j}e_{k}^{(j)}|^{2} \, {dx}
\leq
C\,s^{(j)}(e^{(j)},e^{(j)}).
\]
Therefore, we have
\begin{align*}
\|w-u_{snap}\|^2_{a} & \leq C\sum_i\,s^{(j)}_i(e^{(j)}_i,e^{(j)}_i).
\end{align*}
Finally, by the definition of the eigen-projection, for the case of coupled basis functions, we have
\begin{align*}
s^{(j)}(e^{(j)},e^{(j)})  \leq\cfrac{1}{\lambda_{L_j+1}^{(j)}} \, a^{(j)}_{Q}(e^{(j)},e^{(j)})
  \leq\cfrac{1}{\lambda_{L_j+1}^{(j)}} \, a^{(j)}_{Q}(u_{snap},u_{snap})\leq \cfrac{1}{\lambda_{L_j+1}^{(j)}} \, a^{(j)}_{Q}(u,u)
\end{align*}
and, for the case of un-coupled basis functions, we have
\begin{align*}
s^{(j)}_{i}(e^{(j)}_i,e^{(j)}_i) \leq\cfrac{1}{\lambda_{L_j+1,i}^{(j)}} \, a^{(j)}_i(e^{(j)}_i,e^{(j)}_i)
  \leq\cfrac{1}{\lambda_{L_j+1,i}^{(j)}} \, a^{(j)}_i(u_{snap,i},u_{snap,i})\leq \cfrac{1}{\lambda_{L_j+1,i}^{(j)}} \, a^{(j)}_i(u,u).
\end{align*}
This completes the proof.
\end{proof}

Finally, by using arguments similar as {in} the proof of Lemma \ref{lemma1}, we can prove the following lemma.

\begin{lemma}
\label{lemma3}
Let $u$, $u_{snap}${,} and $w$ be defined as in Theorem \ref{thm1} and \ref{thm2}. For the case of un-coupled basis functions,
we have
\[
\|w(0,\cdot)-u_{snap}(0,\cdot)\|_{c}^{2}\leq\cfrac{CE}{\Lambda_{1}}\|u(0,x)\|_{a}^{2}.
\]
For the case of coupled basis functions, we have
\[
\|w(0,\cdot)-u_{snap}(0,\cdot)\|_{c}^{2}\leq\cfrac{CE}{\Lambda_{2}}\|u(0,x)\|_{a_{Q}}^{2}.
\]
\end{lemma}

\end{document}